\documentclass{amsart}
\usepackage{amssymb,delimset,siunitx,booktabs}

\usepackage[normalem]{ulem}

\usepackage[mathscr]{euscript}
\usepackage[ruled,vlined,linesnumbered]{algorithm2e}
\usepackage{enumitem}
\setlist{itemsep=4pt, topsep=0pt, leftmargin=17pt}

\usepackage{mathtools}

\usepackage{subcaption}
\usepackage{xurl}

\usepackage{xcolor}

\definecolor{PKU}{cmyk}{0, 1, 1, .45}
\definecolor{BIT}{cmyk}{1, 0, 1, 0}
\usepackage[colorlinks, citecolor=BIT, linkcolor=PKU, pagebackref]{hyperref}
\newcolumntype{L}{>{$}l<{$}} 
\newcolumntype{R}{>{$}r<{$}} 
\newcolumntype{C}{>{$}c<{$}} 


\usepackage[capitalize]{cleveref}
\crefname{itm}{}{}
\creflabelformat{itm}{~\upshape(#2#1#3)}
\crefname{def}{Def.}{Defs.}
\Crefname{def}{Definition}{Definitions}
\creflabelformat{def}{~\upshape(#2#1#3)}
\crefname{ineq}{Ineq.}{Ineqs.}
\Crefname{ineq}{Inequality}{Inequalities}
\creflabelformat{ineq}{~\upshape(#2#1#3)}
\crefname{step}{Step}{Step}
\creflabelformat{step}{~\upshape(#2#1#3)}

\AddToHook{env/conjecture/begin}{\crefalias{theorem}{conjecture}}
\AddToHook{env/lemma/begin}{\crefalias{theorem}{lemma}}
\AddToHook{env/proposition/begin}{\crefalias{theorem}{proposition}}

\makeatletter
\newcommand\creflabel[2][\@currentcounter]{%
 \crefalias{\@currentcounter}{#1}\label{#2}}
\makeatother

\newtheorem{theorem}{Theorem}[section]
\newtheorem{conjecture}[theorem]{Conjecture}
\newtheorem{lemma}[theorem]{Lemma}

\numberwithin{equation}{section}

\allowbreak
\allowdisplaybreaks

\usepackage[numbers, sort&compress, nonamebreak, merge, elide, longnamesfirst]{natbib}

\crefname{figure}{FIGURE}{FIGURES}
\Crefname{figure}{FIGURE}{FIGURES}

\usepackage{tikz}
\usetikzlibrary{decorations.pathreplacing, calc, positioning, arrows.meta}
\tikzset{
vertex/.style={shape=circle, minimum size=1mm, fill=black, inner sep=0pt},
edge/.style={black, very thick},
ball/.style={shape=circle, ball color=black, minimum size=1mm, inner sep=0.5},
ellipsis/.style={shape=circle, fill, inner sep=.5},
range.lr/.style={|{Stealth}-{Stealth}|, thin},
range.l/.style={-{Stealth}|, thin},
range.r/.style={|{Stealth}-, thin},
EEdge/.style={very thick, color=black},
apball/.style={ball color=black}
}

\author[K. Zhang]{K. Zhang}
\address[K. Zhang]{School of Mathematics and Statistics, Beijing Institute of Technology, Beijing 102400, P.\ R.\ China.}
\email{kai@bit.edu.cn}

\keywords{Maximal matching,
Average size of maximal matching,
bicyclic graphs}
\subjclass[2020]{05C70, 05C35.}

\title{Extremal graphs for average size of maximal matchings in bicyclic graphs}

\begin{document}

\begin{abstract}
For a graph \(G\), let $avm(G)$ denote the average size
of its maximal matchings. This parameter was introduced by Engbers and
Erey in the study of extremal problems for maximal matchings, and they
asked for extensions from trees and unicyclic graphs to \(k\)-cyclic
graphs. In this paper, we solve the first non-unicyclic case by
determining the minimum value of $avm(G)$ over all
connected bicyclic graphs with \(n\) vertices and \(n+1\) edges. We prove
that, for every connected bicyclic graph \(G\) of order \(n\ge 5\),
\[
\operatorname{avm}(G)\ge \frac{4n-11}{2n-5}.
\]
Moreover, equality holds uniquely for the graph obtained from two
triangles sharing a common edge by attaching all remaining \(n-4\)
pendant edges to one of the two vertices of degree \(3\). The key point
is to translate the minimization of \(\operatorname{avm}(G)\) into
structural restrictions on small maximal matchings, which are then
analyzed through the three possible bicyclic core types.
\end{abstract}
\maketitle
\tableofcontents

\section{Introduction}

\maketitle
Throughout this paper, all graphs considered are finite and simple. At the beginning, we introduce some basic concepts in graph theory that are needed. For more details, please refer to \cite{Bondy2008}.

Let $G = (V, E)$ denote a simple connected graph with order $|V(G)|$ and size $|E(G)|$. In graph theory, a graph $G$ is called a $c$-cyclic graph if $|E(G)| = |V(G)| - 1 + c$. Specifically, when $c$ = 0, 1, 2, $G$ is respectively recognized as a tree, a unicyclic graph, and a bicyclic graph. A leaf vertex $l_u$ of a graph is a vertex of degree 1 incident to a vertex $u\in V(G)$. A vertex is called a support vertex if it has a leaf neighbor. A pendant edge is an edge containing a leaf vertex. Given a support vertex that has degree 2, a support edge is an edge containing the support vertex that is not a pendant edge. A matching $M$ of a graph $G$ is a subset of edges such that no two edges share a common vertex. If $M$ is contained in no other matchings, then it is called a maximal matching. A matching \( M \) of \( G \) is called a maximum matching if for any matching $M'$ of $G$, $|M| \geq |M'|$. A fundamental fact is that a maximum matching is always a maximal matching, but the converse is not true. For a vertex $v \in V(G) $, if $v$ is incident with an edge in $M$, then we say that $M$ covers $v$. A matching $M$ of $G$ is called a perfect matching if all vertices are covered by $M$.  Let \(e, f \in E(G)\) be two edges. We say that \(e\) dominates \(f\) if they share at least one endpoint, i.e., \(e \cap f \neq \varnothing\). For a set of edges \(S \subseteq E(G)\), \(S\) is called an edge dominating set of \(G\) if every edge of \(G\) either belongs to \(S\) or is dominated by some edge in \(S\).

In recent years, the study of the average size of matchings (including ordinary matchings and maximal matchings) has attracted increasing attention. Andriantiana et al.~\cite{Andriantiana2020b} investigated the average size of ordinary matchings and showed that among trees, stars achieve the minimum value while paths achieve the maximum; they also studied the average size of independent vertex sets~\cite{Andriantiana2020a}. Recently, extremal problems concerning maximal matchings have also attracted attention. Maximal matchings hold structural significance; for instance, a maximal matching in $G$ corresponds exactly to an independent dominating set in the line graph $L(G)$~\cite{Engbers2023}. To rigorously capture the expected behavior of such saturated states, Engbers and Erey~\cite{Engbers2023} formally introduced a new parameter, the average size of maximal matchings, denoted by $avm(G)$, defined as the arithmetic mean of the size of all maximal matchings of $G$.

Let \(\mathcal{M}(G)\) be the set of all maximal matchings of \(G\), and
let \(m_i(G)\) denote the number of maximal matchings of size \(i\).
Set
\[
m(G)=|\mathcal{M}(G)|=\sum_i m_i(G)
\]
and
\[
m'(G)=\sum_{M\in \mathcal{M}(G)} |M|
      =\sum_i i\,m_i(G).
\]
The average size of maximal matchings of \(G\), denoted by
\(avm(G)\), is defined by
\[
avm(G)
=
\frac{1}{|\mathcal{M}(G)|}
\sum_{M\in\mathcal{M}(G)} |M|
=
\frac{m'(G)}{m(G)}.
\]


The investigation into the extremal values of $avm(G)$ has proven to be mathematically intricate due to the complex behavior of fractional combinatorial terms. Engbers and Erey~\cite{Engbers2023} posed several interesting questions for investigation, including:

\noindent\textbf{Question 1.} Which graphs $G$ have the maximum or minimum value of $avm(G)$ when $G$ is restricted to a particular family?

Engbers and Erey \cite{Engbers2023} characterized the extremal graphs for $avm(G)$ in the families of trees and connected unicyclic graphs, trees with diameter $d\le 5$, and 2-regular graphs. They conjectured that the graphs minimizing the average size of maximal matchings among trees with fixed order and diameter are star-like trees. The description of the conjecture is as follows.

\begin{conjecture}\cite{Engbers2023}
Let $T$ be a tree of order $n$ with diameter $d$, where $d\neq 4$ and $d\ge 3$. Then, $avm(T)$ is uniquely minimized by $S_{n}(1^{n-d}, (d-1)^{(1)})$.
\end{conjecture} 

Building upon this foundation, Sui and Li \cite{Sui2025} confirmed this conjecture for diameters 6, 7 and 8. In \cite{Engbers2023}, Engbers and Erey posed another question about the average size of maximal matchings.

\noindent\textbf{Question 2.} Extend their results on $avm(G)$ from unicyclic graphs to $k$-cyclic graphs with $k\ge 2$.


In this paper, we make progress on Question 2 by focusing on the case $k=2$. Specifically, we investigate the average size of maximal matchings of $(n, n+1)$-graphs, i.e., connected simple graphs with $n$ vertices and $n+1$ edges. We characterize the $(n, n+1)$-graph with the smallest average size of maximal matchings among all $(n, n+1)$-graphs. Let $\mathcal{G}(n, n+1)$ denote the set of all simple connected graphs with $n$ vertices and $n+1$ edges. The base of a connected bicyclic graph is one of the following three types: two cycles sharing a path, two cycles sharing exactly one vertex, or two vertex-disjoint cycles connected by a path. Deng \cite{Deng2008} demonstrated that for any graph $G \in \mathcal{G}(n, n+1)$, there are two induced cycles $C_p$ and $C_q$ in $G$. Deng \cite{Deng2008} divided all the $(n, n+1)$-graphs with two cycles of lengths $p$ and $q$ into three classes:
\begin{itemize}
\item $\mathcal{A}(p, q, l)$ is the set of $G \in \mathcal{G}(n, n+1)$ in which the cycles $C_p$ and $C_q$ have a common path of length $l$.

\item $\mathcal{B}(p, q)$ is the set of $G \in \mathcal{G}(n, n+1)$ in which the cycles $C_p$ and $C_q$ have only one common vertex;
    
\item $\mathcal{C}(p, q)$ is the set of $G \in \mathcal{G}(n, n+1)$ in which the cycles $C_p$ and $C_q$ have no common vertex.
\end{itemize}

Note that the induced subgraph of vertices on the cycles of $G \in \mathcal{A}(p, q, l)$  (or $\mathcal{B}(p, q)$) is shown in \cref{core} (A) (or (B)). We refer to the induced subgraph formed by the vertices on the cycles as the core of graph $G$. For \(G \in \mathcal{C}(p, q)\), the induced subgraph on the vertices lying on the cycles consists of two disjoint cycles. We also draw the path connecting the two cycles, and the induced subgraph formed by the cycles together with this path is regarded as the core of \(\mathcal{C}(p, q)\), see \cref{core} (C). A pure core maximal matching of $G$ is a maximal matching of $G$ containing only edges of the core. Such a matching must cover every core vertex that is incident with a pendant edge.

\begin{figure}[h]
\centering
\subcaptionbox{The core of $\mathcal{A}(p, q, l)$\label{apql}}[0.28\textwidth]{
\begin{tikzpicture}[scale=0.85]
\node[vertex] (u) at (-1.5,0) {};
\node[vertex] (q) at (-1.0,0) {};
\node[vertex] (w) at (1.0,0) {};
\node[vertex] (v) at (1.5,0) {};
\draw[edge] (u) -- (q);
\draw[edge] (w) -- (v);
\node[vertex] (x) at (-0.9,0.7) {};
\node[vertex] (y) at (0.9,0.7) {};
\node[vertex] (e) at (0.9,-0.7) {};
\node[vertex] (r) at (-0.9,-0.7) {};
\draw[edge] (u) -- (x);
\draw[edge] (u) -- (r);
\draw[edge] (v) -- (e);
\draw[edge] (v) -- (y);
\draw[thick] (-0.9,0.7)--(-0.5,0.7);
\draw[thick] (0.9,0.7)--(0.5,0.7);
\draw[thick] (-1,0)--(-0.6,0);
\draw[thick] (1,0)--(0.6,0);
\draw[thick] (-0.9,-0.7)--(-0.5,-0.7);
\draw[thick] (0.9,-0.7)--(0.5,-0.7);
\node at (0,0) {$\cdots$};
\node at (0,0.7) {$\cdots$};
\node at (0,-0.8) {$\cdots$};
\end{tikzpicture}
}
\hspace{1em}
\subcaptionbox{The core of $\mathcal{B}(p, q)$ \label{bpq}}[0.28\textwidth]{
\begin{tikzpicture}[scale=0.85]
\node[vertex] (c) at (0,0) {};
\node[vertex] (l1) at (-1.3, -1) {};
\node[vertex] (l2) at (1.3, -1) {};
\node[vertex] (l3) at (-0.8, -1) {};
\node[vertex] (l4) at (0.8, -1) {};
\draw[thick] (-0.8, -1) -- (-0.4, -1);
\draw[thick] (0.8, -1) -- (0.4, -1);
\draw[edge] (c) -- (l1) -- (l3);
\draw[edge] (c) -- (l2) -- (l4);
\node[vertex] (r1) at (-1.3, 1) {};
\node[vertex] (r2) at (1.3, 1) {};
\node[vertex] (r3) at (-0.8, 1) {};
\node[vertex] (r4) at (0.8, 1) {};
\draw[thick] (-0.8, 1) -- (-0.4, 1);
\draw[thick] (0.8, 1) -- (0.4, 1);
\draw[edge] (c) -- (r1) -- (r3);
\draw[edge] (c) -- (r2) -- (r4);
\node at (0,1) {$\cdots$};
\node at (0,-1) {$\cdots$};
\end{tikzpicture}
}
\hspace{1em}
\subcaptionbox{The core of $\mathcal{C}(p, q)$\label{cpq}}[0.28\textwidth]{
\begin{tikzpicture}[scale=0.85]
%
%
%
\node[vertex] (v3) at (-0.7,0)  {};
\node[vertex] (v6) at (-0.2,0)  {};
\node[vertex] (v7) at (1.2,0)  {};
\draw[thick] (v3)--(0.1,0);
\node[vertex] (v1) at (-1.5, 1)  {};
\node[vertex] (v2) at (-1.5, -1) {};
\node[vertex] (v4) at (-1.5, 0.6)  {};
\node[vertex] (v5) at (-1.5, -0.6) {};
\draw[thick] (-1.5, 0.6)--(-1.5, 0.3);
\draw[thick] (-1.5, -0.6)--(-1.5, -0.3);
\node[vertex] (w1) at (1.7,0)  {};
\node[vertex] (w2) at (2.5, 1) {};
\node[vertex] (w3) at (2.5, -1) {};
\node[vertex] (w4) at (2.5, 0.6) {};
\node[vertex] (w5) at (2.5, -0.6){};
\draw[thick] (2.5, 0.6)--(2.5, 0.3);
\draw[thick] (2.5, -0.6)--(2.5, -0.3);
\draw[thick] (0.9,0)--(w1);
\draw[edge] (v3) -- (v1) -- (v4);
\draw[edge] (v3) -- (v2) -- (v5);
\draw[edge] (w1) -- (w2) -- (w4);
\draw[edge] (w1) -- (w3) -- (w5);
\node at (0.5,0) {$\cdots$};
\node at (-1.5,0.1) {$\vdots$};
\node at (2.5,0.1) {$\vdots$};
\end{tikzpicture}
}
\caption{}\label{core}
\end{figure}

By relying on this established tripartite classification, we can methodically evaluate the average size of maximal matchings by dissecting the unique leaf-attachment dynamics within each specific core topology. In this paper, our main objective is to characterize the lower bound of the average size of maximal matchings for the family of connected bicyclic graphs $\mathcal{G}(n, n+1)$. By systematically analyzing the topological configurations of bicyclic cores and the leaf-attachment dynamics, we identify the unique extremal graphs that minimize $avm(G)$ among all graphs of order $n$ in $\mathcal{G}(n, n+1)$.


Our main result is the following theorem.

\begin{theorem}\label{thm:main}
Let \(G\in \mathcal{G}(n,n+1)\) be a connected bicyclic graph with
\(n\ge 5\). Then
\[
avm(G)\ge \frac{4n-11}{2n-5}.
\]
Moreover, equality holds if and only if
\[
G\cong \theta_n^1(3,3),
\]
where all \(n-4\) pendant edges are attached to one of the two vertices
of degree \(3\) in the core of \(\theta_n^1(3,3)\).
\end{theorem}

The proof proceeds by considering the three core types separately. In
Section~2, we determine the extremal graph in the class \(\mathcal{A}(p,q,l)\).
Sections~3 and~4 deal with the classes \(\mathcal{B}(p,q)\) and \(\mathcal{C}(p,q)\),
respectively. Finally, in Section~5, we compare the three classwise
extremal values and complete the proof of Theorem~\ref{thm:main}.

\section{The graph with the smallest average size of maximal matchings in $\mathcal{A}(p, q, l)$}
In this section, we will find the $(n, n+1)$-graph with the smallest average size of maximal matchings in $\mathcal{A}(p, q, l)$.

Let $\theta_n^l(p, q)$ be the graph obtained from the graph in \cref{core}(A) by attaching $k=n+1+l-(p+q)$ pendant edges to one of its vertices with degree 3 (see \cref{nlpq}(A)). In \cref{nlpq}(B), the core of $\theta_n^1(3, 3)$ is the induced subgraph formed by vertices $u, v, x, y$.

\begin{figure}[h]
\centering
\subcaptionbox{$\theta_n^l(p, q)$}[0.3\textwidth]{
\begin{tikzpicture}[scale=0.8]
\node[vertex] (u) at (-1.5,0) {};
\node[vertex] (q) at (-1,0) {};
\node[vertex] (w) at (1,0) {};
\node[vertex] (v) at (1.5,0) {};
\draw[edge] (u) -- (q);
\draw[edge] (w) -- (v);
\node at (0,0.3) [label=above:$\dots$] {};
\node at (0,-0.3) [label=above:$\dots$] {};
\node at (0,-0.9) [label=above:$\dots$] {};
\node[vertex] (x) at (-0.5,0.6)  {};
\node[vertex] (y) at (0.5,0.6) {};
\node[vertex] (e) at (0.5,-0.6)  {};
\node[vertex] (r) at (-0.5,-0.6) {};
\draw[edge] (u) -- (x);
\draw[edge] (u) -- (r);
\draw[edge] (v) -- (e);
\draw[edge] (v) -- (y);
\node at (-1.5,-0.5) [label=below:$\dots$] {};
\node[vertex] (u1) at (-2,-0.8) {};
\node[vertex] (u2) at (-1,-0.8) {};
\draw[edge] (u) -- (u1);
\draw[edge] (u) -- (u2);
\end{tikzpicture}
}
\hspace*{1em}
\subcaptionbox{$\theta_n^1(3, 3)$}[0.3\textwidth]{
\begin{tikzpicture}[scale=0.8]
\node[vertex] (u) at (-1.5,0) [label=left:$u$] {};
\node[vertex] (v) at (1.5,0) [label=right:$v$] {};
\node[vertex] (x) at (0,0.8) [label=above:$x$] {};
\node[vertex] (y) at (0,-0.8) [label=below:$y$] {};
\node at (-1.5,-0.5) [label=below:$\dots$] {};
\node[vertex] (u1) at (-2,-0.8) {};
\node[vertex] (u2) at (-1,-0.8) {};
\draw[edge] (u) -- (v);
\draw[edge] (u) -- (u1);
\draw[edge] (u) -- (u2);
\draw[edge] (u) -- (x) -- (v);
\draw[edge] (u) -- (y) -- (v);
\end{tikzpicture}
}
\caption{$\theta_n^l(p, q)$ and $\theta_n^1(3, 3).$}\label{nlpq}
\end{figure}
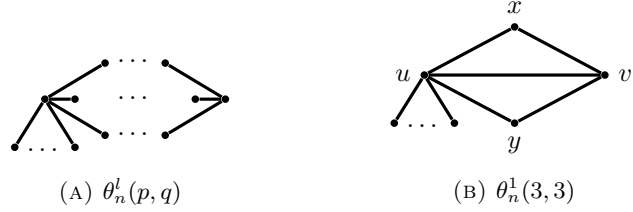

%
%

\begin{theorem}\label{thm:A}
Let $G\in\mathcal{A}(p, q, l)$. Then $avm(G)\ge avm(G^1)$ with the equality if and only if  $G\cong G^1$, where $G^1=\theta_n^1(3, 3)$ is the graph in \cref{nlpq}(B).
\end{theorem}

\begin{proof}Since there is 1 maximal matching with size 1 ($(u,v)$) and $2(n-3)$ maximal matchings with size 2, $k=n-4$, giving
\[
avm(G^1)= \frac{4n-11}{2n-5} = 2 - \frac{1}{2k+3}<2.
\]


If $p \ge 4$, $q \ge 4$, or $l \ge 2$, a maximal matching of size 1 exists iff its unique edge dominates all edges, then no single edge can dominate the core, meaning every maximal matching has size at least 2, and so $avm(G) \ge 2$. Thus, to achieve $avm(G) < 2$, we must have $p=3, q=3$, and $l=1$, which means $G$ contains the core of $\theta_n^1(3,3)$. If $G$ contains a support vertex that is not on the core, then again every maximal matching has size at least 2. Suppose the $k$ leaves are attached to the four core vertices: $a$ leaves on $u$, $b$ on $v$, $c$ on $x$, and $d$ on $y$, with $a+b+c+d = k$. We exhaustively classify the attachments.

\noindent\textbf{Claim.} If $c>0$ or $d>0$, then $avm(G)\ge 2$.

Suppose $c>0$, any maximal matching $M$ must dominate $x$ and its leaves. If $M$ contains a pendant edge, the remaining core vertices form a cycle $C_3$, requiring exactly one additional edge. If $M$ contains no pendant edges, it must select a core edge incident to $x$ to cover it. The remaining uncovered vertices in the core will still require at least one additional edge to be maximally matched. Consequently, $|M| \ge 2$ and $avm(G) \ge 2$.

Since the case where all $k$ leaves are attached to a single vertex $u$ corresponds exactly to $G^1$, to prove the uniqueness of the extremal graph, we only need to consider the case where the leaves are distributed between both $u$ and $v$ (i.e., $a,b\ge1, a+b=k$).

\begin{itemize}
    \item \textit{Select both $(u,l_u), (v,l_v)$ ($ab$ ways):} $x, y$ isolated. Size of maximal matching: 2.
    \item \textit{Select $(u,l_u)$ ($a$ ways):} Must cover $v$ using $(v,x)$ or $(v,y)$. Size of maximal matching: 2. (Total $2a$ maximal matchings).
    \item \textit{Select $(v,l_v)$ ($b$ ways):} By symmetry, $2b$ maximal matchings of size 2.
    \item \textit{Pure Core:} $\{(u,v)\}$ (Size of maximal matching: 1); $\{(u,x), (v,y)\}$ (Size of maximal matching: 2); $\{(u,y), (v,x)\}$ (Size of maximal matching: 2).
\end{itemize}
\[
\begin{cases}
m(G) = ab + 2a + 2b + 1 + 2 = ab + 2k + 3,\\[4pt]
m'(G) = 2ab + 2(2a) + 2(2b) + 1 + 4 = 2ab + 4k + 5,
\end{cases}
\]
thus $avm(G)= 2 - \dfrac{1}{ab+2k+3}$.
Since $a,b \ge 1$, we have $ab > 0$, making the denominator strictly larger than $2k+3$. Hence, $avm(G) > 2 - \frac{1}{2k+3} = avm(G^1)$.This completes this proof.
\end{proof}

\section{The graph with the smallest average size of maximal matchings in $\mathcal{B}(p, q)$}
In this section, we will find the $(n, n+1)$-graph with the smallest average size of maximal matchings in $\mathcal{B}(p, q)$. Let $R_n(p, q)$ be a graph in $\mathcal{B}(p, q)$ such that $k=n+1-(p+q)\ge 1$ pendant edges are attached to a single vertex of degree 2. In \cref{bpq}(B), the core of $\mathcal{B}(3, 3)$ is the induced subgraph formed by vertices $u, v_1, v_2, w_1, w_2$.
 
\begin{figure}[h]
\subcaptionbox{$R_n(3, 3)$}[0.3\textwidth]{
\begin{tikzpicture}[scale=0.8]
\node[vertex] (c) at (0,0) [label=below:$u$] {};
\node at (-1.9,-0.3) [label=below:$\dots$] {};
\node[vertex] (l1) at (-1, 0.6) [label=above:$v_1$]{};
\node[vertex] (l2) at (-1, -0.6) [label=below:$v_2$]{};
\node[vertex] (r1) at (1, 0.6) [label=above:$w_1$]{};
\node[vertex] (r2) at (1, -0.6) [label=below:$w_2$]{};
\node[vertex] (r3) at (-2.5, -0.6) {};
\node[vertex] (r4) at (-1.5, -0.6) {};
\draw[edge] (c) -- (l1) -- (l2) -- (c);
\draw[edge] (c) -- (r1) -- (r2) -- (c);
\draw[edge] (l1) -- (r3);
\draw[edge] (l1) -- (r4);
\end{tikzpicture}
}
\hspace*{1em}
\subcaptionbox{The core of $\mathcal{B}(3, 3)$}[0.3\textwidth]{
\begin{tikzpicture}[scale=0.8]
\node[vertex] (c) at (0,0) [label=below:$u$] {};
\node[vertex] (l1) at (-1, 0.6) [label=above:$v_1$]{};
\node[vertex] (l2) at (-1, -0.6) [label=below:$v_2$]{};
\node[vertex] (r1) at (1, 0.6) [label=above:$w_1$]{};
\node[vertex] (r2) at (1, -0.6) [label=below:$w_2$]{};
\draw[edge] (c) -- (l1) -- (l2) -- (c);
\draw[edge] (c) -- (r1) -- (r2) -- (c);
\end{tikzpicture}
}
\caption{}\label{bpq}
\end{figure}

\begin{lemma}
Avm$(G^2)=\frac{7n-27}{3n-11}$, where $G^2=R_n(3, 3)$ is the graph in \cref{bpq}(A).
\end{lemma}
\begin{proof}
Since any maximal matching contains either $w_1w_2$ or $uw_1$ (and $uw_2$ is symmetric), we only need to consider the following two cases.

\noindent\textbf{Case 1.} $(u, w_1)$ belongs to the maximal matching. The remaining edges $(u, w_2)$ and $(w_1, w_2)$ are already dominated by $u$ or $w_1$. Edges $(u, v_1)$ and $(u, v_2)$ are dominated by $u$. The only undominated edges are $(v_1,v_2)$ and the pendant edge $(v_1,l_{v_1})$. Since they all share the vertex $v_1$, we must choose exactly one of them to add to the matching to achieve maximality. Choosing $(v_1,v_2)$ yields the maximal matching $\{(u, w_1), (v_1, v_2)\}$ of size $2$, and there is $1$ such maximal matching. Choosing $(v_1, l_{v_1})$ yields the maximal matching $\{(u, w_1), (v_1,l_{v_1})\}$ of size $2$, and there are $k$ such maximal matchings. By symmetry, selecting $(u, w_2)$ gives another $1+k$ maximal matchings. Hence, this case contributes $2k+2$ maximal matchings of size $2$.

\noindent\textbf{Case 2.} $(w_1, w_2)$ belongs to the maximal matching. The three possible choices for the second edge are $(u, v_1)$, $(v_1, v_2)$, or $(u, v_2)$. Choosing $(u, v_1)$ yields the unique maximal matching $\{(w_1, w_2), (u, v_1)\}$ of size $2$; choosing $(v_1, v_2)$ yields the unique maximal matching $\{(w_1, w_2), (v_1, v_2)\}$ also of size $2$; and choosing $(u, v_2)$ forces us to add exactly one pendant edge $(v_1, l_{v_1})$, giving $k$ maximal matchings $\{(w_1, w_2), (u, v_2), (v_1, l_{v_1})\}$ of size $3$. Hence this case contributes 2 maximal matchings of size $2$ and $k$ maximal matchings of size $3$.

\[
\begin{cases}
m(G^2) = (2k+2) + 2 + k = 3k + 4,\\[4pt]
m'(G^2) = 2 \times (2k+4) + 3 \times k = 7k + 8.
\end{cases}
\]

Substituting $k = n-5$, we obtain the average size:
\[
avm(G^2) = \frac{7(n-5)+8}{3(n-5)+4} = \frac{7n-27}{3n-11}=3-\frac{2k+4}{3k+4}=3-\frac{2n-6}{3n-11}.
\]

\end{proof}

\begin{lemma}\label{lem:no-size1}
For any graph $G \in \mathcal{B}(p, q)$, there exists no maximal matching of size $1$.
\end{lemma}

\begin{proof}
Any graph in this class consists of two cycles ($C_p$ and $C_q$) that intersect in exactly one vertex, denoted by $u$. Any single edge covers two vertices. If $e$ is an edge of $C_p$ not incident to $u$, then it cannot dominate the edges of $C_q$ that are also not incident to $u$. If $e=(u, v)$ (assume that $v\in V(C_p)$), it cannot dominate an edge that is not incident to $u$ and $v$ in $C_p$ when $p\ge 4$, and it cannot dominate an edge that is not incident to $u$ in $C_q$. Hence, a single edge can never dominate the whole graph, and therefore no maximal matching can have size $1$, every maximal matching $M$ in $G$ satisfy $|M| \ge 2$.
\end{proof}

\begin{lemma}
\label{lem:no-c4}
For any graph $G \in \mathcal{B}(p, q)$, if $\max \{p, q\} \ge 4$, then $avm (G) > avm(G^2)$.
\end{lemma}

\begin{proof}
By symmetry, we may assume that \(p\ge q\). Since Lemma~\ref{lem:no-size1}
shows that no graph in \(\mathcal{B}(p,q)\) has a maximal matching of size \(1\),
we have
\[
m'(G)\ge 2m_2(G)+3\bigl(m(G)-m_2(G)\bigr)
=3m(G)-m_2(G).
\]
Hence
\[
avm(G)
=
\frac{m'(G)}{m(G)}
\ge
3-\frac{m_2(G)}{m(G)}.
\]
Therefore, throughout the proof, it is enough to control the possible
maximal matchings of size \(2\).

If \(G\) contains a support vertex not lying on the core, then no maximal
matching of size \(2\) exists. Indeed, if such a matching contained an
edge outside the core, then the remaining single edge could not dominate
the whole bicyclic core. If it contained only core edges, then it could
not dominate the pendant edge incident with the support vertex outside
the core. Thus \(m_2(G)=0\), and so
\[
avm(G)\ge 3>avm(G^2).
\]
Hence, in the rest of the proof, we may assume that all pendant edges
are attached directly to core vertices.

\noindent\textbf{Case 1.} \(p\ge4\) and \(q\ge4\).

In this case no maximal matching of size \(2\) exists. To see this,
suppose that \(M\) is a maximal matching of size \(2\). Since each of
the two cycles has length at least \(4\), one edge cannot dominate an
entire cycle. Thus the two edges of \(M\) would have to dominate both
cycles simultaneously. However, an edge lying on one cycle can only help
to dominate the other cycle if it is incident with the common vertex
\(u\), and even then it cannot dominate all edges of a cycle of length
at least \(4\). Hence at least one edge of the core remains undominated,
contradicting the maximality of \(M\). Therefore \(m_2(G)=0\), and
\[
avm(G)\ge 3>avm(G^2).
\]

\medskip
\noindent\textbf{Case 2.} \(p=4\) and \(q=3\).

\begin{figure}[h]
\begin{tikzpicture}[scale=0.8]
\node[vertex] (c) at (0,0) [label=below:$u$] {};
\node[vertex] (l1) at (-1, 0.6) [label=above:$v_1$]{};
\node[vertex] (l2) at (-2, 0) [label=below:$v_2$]{};
\node[vertex] (l3) at (-1, -0.6) [label=below:$v_3$]{};

\node[vertex] (r1) at (1, 0.6) [label=above:$w_1$]{};
\node[vertex] (r2) at (1, -0.6) [label=below:$w_2$]{};

\draw[edge] (c) -- (l1) -- (l2) -- (l3)--(c);
\draw[edge] (c) -- (r1) -- (r2) -- (c);

\end{tikzpicture}
\caption{The core of $\mathcal{B}(4, 3).$}$\label{R34}$
\end{figure}
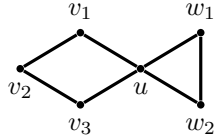

Let $k=n-6$ be the number of pendant edges of the present graph \(G\). If \(k=0\), then \(G\) is exactly the core of \(B(4,3)\). A direct
enumeration gives four maximal matchings of size \(2\) and two maximal
matchings of size \(3\). Hence $m(G)=6, m'(G)=4\cdot 2+2\cdot 3=14$,
and so $avm(G)=\frac{14}{6}=\frac{7}{3}$.

For the same order \(n=6\), the graph \(G^2=R_6(3,3)\) has one pendant
edge, and therefore $avm(G^2)=3-\frac{6}{7}=\frac{15}{7}$. Thus $avm(G)=\frac{7}{3}>\frac{15}{7}
=avm(G^2)$.

Hence, in what follows, we may assume that \(k\ge 1\). Notice that
the same-order graph \(G^2=R_n(3,3)\) has
\[
K=n-5=k+1
\]
pendant edges. Hence
\[
avm(G^2)
=
3-\frac{2K+4}{3K+4}
=
3-\frac{2k+6}{3k+7}.
\]

The core of \(\mathcal{B}(4,3)\) has exactly four maximal matchings of size \(2\):
\[
\begin{aligned}
M_1&=\{(v_1,v_2), (u,w_1)\},&
M_2&=\{(v_1,v_2), (u,w_2)\},\\
M_3&=\{(v_2,v_3),(u,w_1)\},&
M_4&=\{(v_2,v_3), (u,w_2)\}.
\end{aligned}
\]
In addition, a maximal matching of size \(2\) containing a pendant edge
can occur only when the pendant edge is incident with \(v_2\), in which
case the other edge must be either \((u,w_1)\) or \((u, w_2)\).

First suppose that all \(k\) pendant edges are attached to a single core
vertex. Up to symmetry, there are four possibilities. Direct enumeration
of maximal matchings gives the following table:
\[
\begin{array}{c|c|c|c}
\text{attachment vertex} & m_2(G) & m(G) &
avm(G) \\ \hline
v_2 & 2k+4 & 4k+6 &
3-\dfrac{2k+4}{4k+6} \\[2mm]
u & 4 & 2k+6 &
3-\dfrac{4}{2k+6} \\[2mm]
w_1\ \text{or}\ w_2 & 2 & 4k+4 &
3-\dfrac{2}{4k+4} \\[2mm]
v_1\ \text{or}\ v_3 & 2 & 4k+4 &
3-\dfrac{2}{4k+4}
\end{array}
\]
For each row, we compare the subtracted fraction with
\[
\frac{2k+6}{3k+7}.
\]
Indeed,
\[
\frac{2k+4}{4k+6}
<
\frac{2k+6}{3k+7},
\]
because
\[
(2k+6)(4k+6)-(2k+4)(3k+7)
=
2k^2+10k+8>0.
\]
Also,
\[
\frac{4}{2k+6}
<
\frac{2k+6}{3k+7},
\]
because
\[
(2k+6)(2k+6)-4(3k+7)
=
4k^2+12k+8>0.
\]
Finally,
\[
\frac{2}{4k+4}
<
\frac{2k+6}{3k+7}
\]
is immediate for every \(k\ge1\). Therefore, in all single-attachment
configurations,
\[
avm(G)>avm(G^2).
\]

It remains to consider the case where pendant edges are attached to at
least two core vertices. Let \(X\) be the set of core vertices incident
with pendant edges. The following estimates are obtained by direct
inspection of the \(\mathcal{B}(4,3)\) core. They distinguish whether \(v_2\), the
only vertex that can be incident with a pendant edge in a size-\(2\)
maximal matching, belongs to \(X\):
\[
\begin{array}{c|c|c}
\text{condition on }X & \text{upper bound for }m_2(G)
& \text{lower bound for }m(G) \\ \hline
v_2\notin X & m_2(G)\le 4 & m(G)\ge 2k+4 \\[1mm]
X=\{u,v_2\} & m_2(G)=2b+4
& m(G)=(a+4)b+2a+6 \\[1mm]
v_2\in X,\ X\not\subseteq\{u,v_2\}
& m_2(G)\le k+2 & m(G)\ge 3k
\end{array}
\]
In the second row, \(a\ge1\) and \(b\ge1\) denote the numbers of pendant
edges attached to \(u\) and \(v_2\), respectively, so that \(k=a+b\).

If \(v_2\notin X\), then
\[
avm(G)
\ge
3-\frac{4}{2k+4}.
\]
Since
\[
\frac{4}{2k+4}
<
\frac{2k+6}{3k+7},
\]
we obtain
\[
avm(G)>avm(G^2).
\]

If \(X=\{u,v_2\}\), then
\[
avm(G)
\ge
3-\frac{2b+4}{(a+4)b+2a+6}.
\]
Since \(k=a+b\), it is enough to prove
\[
\frac{2b+4}{(a+4)b+2a+6}
<
\frac{2a+2b+6}{3a+3b+7}.
\]
After cross-multiplication, the difference between the right-hand side
numerator and the left-hand side numerator is
\[
\begin{aligned}
&(2a+2b+6)\bigl((a+4)b+2a+6\bigr)
-(2b+4)(3a+3b+7)\\
& =
2a^2b+4a^2+2ab^2+12ab+12a+2b^2+10b+8>0.
\end{aligned}
\]
Thus
\[
avm(G)>avm(G^2).
\]

Finally, suppose that \(v_2\in X\) and \(X\not\subseteq\{u,v_2\}\).
Then
\[
avm(G)
\ge
3-\frac{k+2}{3k}.
\]
Since \(k\ge2\) in this dispersed case, we have $\frac{k+2}{3k}<\frac{2k+6}{3k+7}$, because this is equivalent to $3k^2+5k-14>0$, which holds for all \(k\ge2\). Hence
\[
avm(G)>avm(G^2).
\]
Therefore the conclusion holds for \(\mathcal{B}(4,3)\).

\medskip
\noindent\textbf{Case 3.} \(p=5\) and \(q=3\).

\begin{figure}[h]
\begin{tikzpicture}[scale=0.8]
\node[vertex] (c) at (0,0) [label=below:$u$] {};
\node[vertex] (l1) at (-1, 0.6) [label=above:$v_1$]{};
\node[vertex] (l2) at (-2, 0) [label=left:$v_2$]{};
\node[vertex] (l3) at (-1.5, -0.6) [label=below:$v_3$]{};
\node[vertex] (l4) at (-0.5, -0.6) [label=below:$v_4$]{};

\node[vertex] (r1) at (1, 0.6) [label=above:$w_1$]{};
\node[vertex] (r2) at (1, -0.6) [label=below:$w_2$]{};

\draw[edge] (c) -- (l1) -- (l2) -- (l3)--(l4)--(c);
\draw[edge] (c) -- (r1) -- (r2) -- (c);

\end{tikzpicture}
\caption{The core of $\mathcal{B}(5, 3).$}$\label{R34}$
\end{figure}
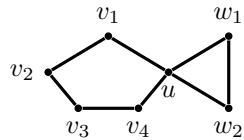

Let $k=n-7$ be the number of pendant edges of \(G\). If \(k=0\), then \(G\) is exactly the core of \(\mathcal{B}(5,3)\). A direct
enumeration gives $m(G)=9, m'(G)=25$, and hence $avm(G)=\frac{25}{9}$.
For the same order \(n=7\), the graph \(G^2=R_7(3,3)\) has two pendant
edges, and $avm(G^2)=\frac{22}{10}=\frac{11}{5}$. Thus $avm(G)=\frac{25}{9}>\frac{11}{5}=avm(G^2)$.
Hence, in what follows, we may assume that \(k\ge1\). Then the same-order graph
\(G^2=R_n(3,3)\) has
\[
K=n-5=k+2
\]
pendant edges, and therefore
\[
avm(G^2)
=
3-\frac{2K+4}{3K+4}
=
3-\frac{2k+8}{3k+10}.
\]

We first determine the maximal matchings of size \(2\). If such a
matching contains the edge \((w_1, w_2)\), then its other edge would have to
dominate the whole \(5\)-cycle, which is impossible. Hence a size-\(2\)
maximal matching must contain either \((u, w_1)\) or \((u, w_2)\). Such an edge
dominates the two edges of the \(5\)-cycle incident with \(u\). The
remaining three consecutive edges $(v_1, v_2)$, $(v_2,v_3)$, $(v_3, v_4)$ can be dominated by exactly one edge, namely \((v_2, v_3)\). Thus the only
possible size-\(2\) core maximal matchings are
\[
N_1=\{(u,w_1),(v_2,v_3)\},
\qquad
N_2=\{(u,w_2),(v_2,v_3)\}.
\]
Moreover, no maximal matching of size \(2\) can contain a pendant edge,
because after choosing a pendant edge, the remaining single edge cannot
dominate the whole \(5\)-cycle together with the triangle. Hence $m_2(G)\le 2$. By direct enumeration of the possible extensions in this core, we have $m(G)\ge k+2$.

Consequently, $avm(G)\ge3-\frac{2}{k+2}$. Since $\frac{2}{k+2}<\frac{2k+8}{3k+10}$, which is equivalent to $2k^2+6k-4>0$, we obtain, for every \(k\ge1\),
\[
avm(G)>
3-\frac{2k+8}{3k+10}
=
avm(G^2).
\]

\medskip
\noindent\textbf{Case 4.} \(p\ge6\) and \(q=3\).

Let the \(p\)-cycle be
\[
uv_1v_2\cdots v_{p-1}u
\]
and let the triangle be \(uw_1w_2u\). We show that no maximal matching
of size \(2\) exists.

Indeed, a size-\(2\) maximal matching must dominate the triangle. Hence
it must contain either \((w_1,w_2)\), \((u,w_1)\), or \((u,w_2)\). If it contains
\((w_1,w_2)\), then the other edge would have to dominate the whole
\(p\)-cycle, which is impossible for \(p\ge4\). If it contains \((u,w_1)\)
or \((u,w_2)\), then this edge dominates only the two edges of the
\(p\)-cycle incident with \(u\). The other edge can dominate at most
three consecutive edges of the remaining path
\[
v_1v_2\cdots v_{p-1}.
\]
Since \(p\ge6\), at least one edge of the \(p\)-cycle remains
undominated. Thus no maximal matching of size \(2\) exists, so $m_2(G)=0$.

Therefore
\[
avm(G)\ge 3>avm(G^2).
\]

Combining Cases \(1\)--\(4\), we conclude that for every
\(G\in \mathcal{B}(p,q)\) with \(\max\{p,q\}\ge4\),
\[
avm(G)>avm(G^2).
\]
This completes the proof.
\end{proof}

\begin{lemma}\label{lem:mixed_dispersion}\label{uinvin}
Let $G\in \mathcal{B}(3, 3)$. Suppose $a \ge 1$ pendant edges are attached to $u$, and the remaining pendant edges are attached to $d$ degree-2 vertices ($1 \le d \le 4$). Then $avm(G) > avm(G^2)$.
\end{lemma}

\begin{proof}
Let $S \subseteq \{v_1, v_2, w_1, w_2\}$ be the set of degree-2 vertices that have at least one leaf attached, where $d = |S|$. For a matching to be maximal and contain no pendant edges, it must cover all vertices that have leaves attached; otherwise, an uncovered pendant edge could be added. Thus, any size-2 maximal matching must cover the set $S \cup \{u\}$, which has size $d+1$. We analyze the maximum possible $m_2$ based on $d$:

\noindent\textbf{Case 1.} $d=4$. 

The set of vertices requiring coverage is $\{u, v_1, v_2, w_1, w_2\}$, which contains 5 vertices. A size-2 maximal matching can cover 4 vertices. Therefore, at least one vertex with leaves remains uncovered, forcing the inclusion of a pendant edge. Thus, no maximal matching of size 2 exists ($m_2 = 0$). Every maximal matching has size at least 3, implying $avm(G) \ge 3 > avm(G^2)$.

\noindent\textbf{Case 2.} $d=3$ or $d=2$.

For $d=3$, without loss of generality, assume $S = \{v_1, v_2, w_1\}$. The maximal matching must cover $\{u, v_1, v_2, w_1\}$. The only combination of 2 core edges covering exactly these 4 vertices is $\{(v_1, v_2), (u, w_1)\}$. (Any other combination leaves at least one of them free). Thus, $m_2 \le 1$. For $d=2$, if $S = \{v_1, v_2\}$, the maximal matching must cover $\{u, v_1, v_2\}$. The valid size-2 maximal matchings are $\{(v_1, v_2), (u, w_1)\}$ and $\{(v_1, v_2), (u, w_2)\}$. If $S = \{v_1, w_1\}$, the maximal matching must cover $\{u, v_1, w_1\}$. The valid size-2 maximal matchings are $\{(u, v_1), (w_1, w_2)\}$ and $\{(u, w_1), (v_1, v_2)\}$. In both configurations, there are at most 2 valid core maximal matchings. Thus, $m_2 \le 2$.

We have $m'(G) \ge 2m_2 + 3(m - m_2) = 3m - m_2$. Therefore, the average size is bounded by:
\[ avm(G) = \frac{m'}{m} \ge 3 - \frac{m_2}{m}. \]
For $d \in \{2, 3\}$, we have $m_2 \le 2$. Furthermore, the total number of maximal matchings $m$ includes at least the selections of one leaf from each attached vertex, ensuring $m \ge k + 1 \ge 4$.
We compare this bound to the extremal candidate $G^2$, whose average is $avm(G^2) = 3 - \frac{2k+4}{3k+4}$.
We require:
\[ 3 - \frac{2}{m} > 3 - \frac{2k+4}{3k+4} \iff \frac{2}{m} < \frac{2k+4}{3k+4} \iff 6k + 8 < m(2k+4). \]
Since $m \ge k+1$, we substitute to get $6k + 8 < (k+1)(2k+4) = 2k^2 + 6k + 4$, which simplifies to $2k^2 > 4$. Since the total leaves $k \ge d+1 \ge 3$, this strict inequality unconditionally holds. Thus, $avm(G) > avm(G^2)$ for $d=3$ or $d=2$.

\noindent\textbf{Case 3.} $d=1$.

This configuration reduces exactly to the mixed attachment scenario where leaves are strictly distributed between $u$ and a vertex $v_1$ with degree 2, specifically, $v_1$ is incident with $b$ pendant edges.
\begin{itemize}
\item \textit{Choose a pendant edge $(u,l_u)$ ($a$ choices)}. Then $u$ is covered. Then $(w_1,w_2)$ in the maximal matching. To cover the $b$ leaves on $v_1$, we may either take $(v_1,v_2)$ or a pendant edge $(v_1, l_{v_1})$. This gives $a$ maximal matchings of size $3$ (with $(v_1,v_2)$) and $ab$ maximal matchings of size $3$ (with $(v_1, l_{v_1})$). Contribution: $a+ab$ maximal matchings of size $3$.

\item \textit{Choose a pendant edge $(v_1, l_{v_1})$ but no pendant edge from $u$ ($b$ choices)}. Then $v_1$ is covered. To dominate the leaves on $u$, we must select an edge incident to $u$ (other than $(u,v_1)$ because $v_1$ is already taken). Three sub‑subcases: Matching $\{(v_1, l_{v_1}),(u,v_2),(w_1,w_2)\}$, size $3$, $b$ maximal matchings. Matching $\{(v_1,l_{v_1}),(u,w_1)\}$, size $2$, $b$ maximal matchings. Matching $\{(v_1,l_{v_1}),(u,w_2)\}$, size $2$, $b$ maximal matchings. Contribution: $b$ of size $3$ and $2b$ of size $2$.

\item \textit{Choose no pendant edge at all}. We must cover both $u$ and $v_1$ using core edges. The possibilities: $\{(u,v_1),(w_1,w_2)\}$, size $2$, $1$ maximal matching; $\{(u,w_1),(v_1,v_2)\}$, size $2$, $1$ maximal matching; $\{(u,w_2),(v_1,v_2)\}$, size $2$, $1$ maximal matching. Contribution: $3$ maximal matchings of size $2$.
\end{itemize}
Summing up:
\[
\begin{cases}
m(G) = ab + a + 3b + 3,\\[4pt]
m'(G) = 2(2b+3) + 3(ab+a+b) = 3ab + 3a + 7b + 6.
\end{cases}
\]
Hence
\[
avm(G)=\frac{3ab+3a+7b+6}{ab+a+3b+3}
=3-\frac{2b+3}{ab+a+3b+3}>avm(G^2), \text{ since }a,b\ge1.
\]

\end{proof}

\begin{lemma}\label{unoinvin}
Let $G \in \mathcal{B}(3, 3)$. Suppose $u$ has no attached pendant edges ($a=0$), and the $k$ pendant edges are attached to $d$ degree-2 vertices ($1 \le d \le 4$). If $d \ge 2$, then $avm(G) > avm(G^2)$.
\end{lemma}
\begin{proof}
Let $S \subseteq \{v_1, v_2, w_1, w_2\}$ be the set of degree-2 vertices that have at least one pendant edge attached, meaning $|S| = d$. We aim to bound the maximum number of size-2 maximal matchings. In the core of $\mathcal{B}(3, 3)$, there are exactly 5 maximal matchings of size 2:
\begin{itemize}
    \item $M_1 = \{(v_1, v_2), (w_1, w_2)\}$, 
    \item $M_2 = \{(u, v_1), (w_1, w_2)\}$, 
    \item $M_3 = \{(u, v_2), (w_1, w_2)\}$, 
    \item $M_4 = \{(v_1, v_2), (u, w_1)\}$, 
    \item $M_5 = \{(v_1, v_2), (u, w_2)\}$. 
\end{itemize}

For a pure core maximal matching to be maximal in $G$ (i.e., containing no pendant edges), it is strictly required to cover all vertices in $S$; otherwise, an uncovered pendant edge could be legally added, violating maximality. We filter these 5 maximal matchings based on their intersection with $S$:

If $d=4$. then we have $S = \{v_1, v_2, w_1, w_2\}$. Only $M_1$ successfully covers all 4 vertices in $S$. Thus, $m_2 = 1$. 

If $d=3$. Without loss of generality, let $S = \{v_1, v_2, w_1\}$. The maximal matchings covering $S$ are $M_1$ and $M_4$. Thus, $m_2 = 2$.

If $d=2$.
If $S = \{v_1, v_2\}$, they are covered by $M_1, M_4$, and $M_5$, $m_2 = 3$. If $S = \{v_1, w_1\}$, they are covered by $M_1, M_2$, and $M_4$, $m_2 = 3$. 

In all dispersion configurations ($d \ge 2$), we universally have $m_2 \le 3$. The total size is $m' \ge 2m_2 + 3(m - m_2) = 3m - m_2$. Thus, the average size is bounded by:
\[ avm(G) = \frac{M}{m} \ge 3 - \frac{m_2}{m} \ge 3 - \frac{3}{m}. \]

We now establish a loose lower bound for $m$. Since $d \ge 2$, the $k$ leaves are partitioned into at least two non-empty sets. Selecting any single pendant edge guarantees a valid maximal matching (at least $k$ maximal matchings). Selecting a pair of pendant edges from different vertices in $S$ guarantees a maximal matching of size $\ge 3$. Even in the minimal dispersion where $b=1$ and $c=k-1$, the number of leaf-inclusive maximal matchings is at least $(1)(k-1) + k = 2k-1$. Including the pure core maximal matchings ($m_2 \ge 1$), we safely have $m \ge 2k$. To prove $avm(G) > avm(G^2) = 3 - \frac{2k+4}{3k+4}$, it suffices to show:
\[ 3 - \frac{3}{m} > 3 - \frac{2k+4}{3k+4} \iff \frac{3}{m} < \frac{2k+4}{3k+4} \iff 9k + 12 < m(2k+4). \]
Substituting our bound $m \ge 2k$:
\[ m(2k+4) \ge 2k(2k+4) = 4k^2 + 8k. \]
We require $9k + 12 < 4k^2 + 8k$, which simplifies to $4k^2 - k - 12 > 0$.
Since $d \ge 2$, we must have $k \ge 2$ pendant edges. For the boundary case $k=2$, we get $4(4) - 2 - 12 = 2 > 0$. The strict inequality holds universally for all $k \ge 2$.


Therefore, any dispersion of pendant edges among multiple degree-2 vertices strictly increases the average size of maximal matchings compared to $G^2$.
\end{proof}

\begin{theorem}\label{thm:B}
For any graph $G \in \mathcal{B}(p, q)$ and $G\ncong G^2$, $avm (G) > avm(G^2)$.
\end{theorem}

\begin{proof}
If $G$ contains a support vertex that is not on the core of $\mathcal{B}_{p,q}$, then every maximal matching has size at least 3. By Lemmas \ref{lem:no-c4}, \ref{uinvin}, \ref{unoinvin}, we only need to prove that $G \in \mathcal{B}(3, 3)$ and all pendant edges are attached to $u$.


Pendant edges $(u, l_u)$ are chosen. There are $k$ choices. The remaining edges are $(v_1,v_2)$ and $(w_1,w_2)$, which must both be added to the maximal matching. Hence each such choice yields the maximal matching $\{(u,l_u), (v_1,v_2), (w_1,w_2)\}$ of size $3$. Contribution: $k$ maximal matchings of size $3$.

If $(u,v_1)$ is chosen, then $u$ is covered, dominating all pendant edges and edges incident to $u$; only $(w_1,w_2)$ remains free and must be selected. This yields $\{(u,v_1),(w_1,w_2)\}$ of size $2$. Similarly $(u,v_2)$ gives $\{(u,v_2),(w_1,w_2)\}$ of size $2$, $(u,w_1)$ gives $\{(u,w_1),(v_1,v_2)\}$, and $(u,w_2)$ gives $\{(u,w_2),(v_1,v_2)\}$. Contribution: $4$ maximal matchings of size $2$.

No other maximal matchings exist.
\[
m(G)=k+2+2=k+4=n-1,\qquad
m'(G)=3k+4\cdot2=3k+8=3n-7.
\]
Hence
\[
avm(G)=\frac{3n-7}{n-1}=3-\frac{4}{n-1}.
\]
Comparing with $avm(G^2)$ for $k=n-5$: For $n\ge5$, one checks $\frac{4}{n-1}<\frac{2n-6}{3n-11}$, thus $avm(G)>avm(G^2)$.

\end{proof}

\section{The graph with the smallest average size of maximal matchings in $\mathcal{C}(p, q)$}

In this section, we will find the $(n, n+1)$-graph with the smallest average size of maximal matchings in $\mathcal{C}(p, q)$. Let $T_n^r(p, q)$ be the $(n, n+1)$-graph obtained by connecting $C_p$ and $C_q$ by a path of length $r$ and the other $k=n+1-p-q-r\ge1$ edges are all attached to the common vertex of the path and $C_p$ (see \cref{cpq}(A)).

\begin{figure}[h]
\subcaptionbox{$T_n^r(p, q)$}[0.3\textwidth]{
\begin{tikzpicture}[scale=0.8]
\node[vertex] (v3) at (-0.7,0)  {};
\node[vertex] (v6) at (-0.2,0)  {};
\node[vertex] (v7) at (1.2,0)  {};
\draw[thick] (v3)--(0.1,0);
\node[vertex] (v1) at (-1.5, 1)  {};
\node[vertex] (v2) at (-1.5, -1) {};
\node[vertex] (v4) at (-1.5, 0.6)  {};
\node[vertex] (v5) at (-1.5, -0.6) {};
\draw[thick] (-1.5, 0.6)--(-1.5, 0.3);
\draw[thick] (-1.5, -0.6)--(-1.5, -0.3);
\node[vertex] (w1) at (1.7,0)  {};
\node[vertex] (w2) at (2.5, 1) {};
\node[vertex] (w3) at (2.5, -1) {};
\node[vertex] (w4) at (2.5, 0.6) {};
\node[vertex] (w5) at (2.5, -0.6){};
\node[vertex]  at (0, -1) {};
\node[vertex]  at (-1, -1){};
\draw[thick] (v3)--(0, -1);
\draw[thick]  (v3)--(-1, -1);
\draw[thick] (2.5, 0.6)--(2.5, 0.3);
\draw[thick] (2.5, -0.6)--(2.5, -0.3);
\draw[thick] (0.9,0)--(w1);
\draw[edge] (v3) -- (v1) -- (v4);
\draw[edge] (v3) -- (v2) -- (v5);
\draw[edge] (w1) -- (w2) -- (w4);
\draw[edge] (w1) -- (w3) -- (w5);
\node at (0.5,0) {$\cdots$};
\node at (-0.5,-1) {$\cdots$};
\node at (-1.5,0.1) {$\vdots$};
\node at (2.5,0.1) {$\vdots$};
\end{tikzpicture}
}
\hspace*{1em}
\subcaptionbox{$T_n^1(3, 3)$}[0.3\textwidth]{
\begin{tikzpicture}[scale=0.8]
\node[vertex] (v3) at (-0.5,0) [label=above:$v_3$] {};
\node[vertex] (v1) at (-1.5, 0.6) [label=above:$v_1$] {};
\node[vertex] (v2) at (-1.5, -0.6)[label=below:$v_2$] {};
\node[vertex] (w1) at (1.5,0) [label=above:$w_1$] {};
\node[vertex] (w2) at (2.5, 0.6)[label=above:$w_2$] {};
\node[vertex] (w3) at (2.5, -0.6)[label=below:$w_3$] {};
\draw[edge] (v3) -- (v1) -- (v2) -- (v3);
\draw[edge] (v3) --  (w1);
\draw[edge] (w1) -- (w2) -- (w3) -- (w1);
\node[vertex] (a3) at (-1,-0.6)  {};
\node[vertex] (a1) at (0.1, -0.6)  {};
\node(a2) at (-0.4, -0.3)[label=below:$\dots$] {};
\draw[edge] (v3) --  (a1);
\draw[edge] (v3) -- (a3);
\end{tikzpicture}
}
\hspace*{1em}
\subcaptionbox{The core of $T_n^1(3, 3)$}[0.3\textwidth]{
\begin{tikzpicture}[scale=0.8]
\node[vertex] (v3) at (-0.5,0) [label=above:$v_3$] {};
\node[vertex] (v1) at (-1.5, 0.6) [label=above:$v_1$] {};
\node[vertex] (v2) at (-1.5, -0.6)[label=below:$v_2$] {};
\node[vertex] (w1) at (1.5,0) [label=above:$w_1$] {};
\node[vertex] (w2) at (2.5, 0.6)[label=above:$w_2$] {};
\node[vertex] (w3) at (2.5, -0.6)[label=below:$w_3$] {};
\draw[edge] (v3) -- (v1) -- (v2) -- (v3);
\draw[edge] (v3) --  (w1);
\draw[edge] (w1) -- (w2) -- (w3) -- (w1);
\end{tikzpicture}
}
\caption{}\label{cpq}
\end{figure}

\begin{lemma}\label{avmg3}
Let \(G^3=T_n^1(3,3)\) with \(k=n-6\ge1\) pendant edges.
Then
\[
avm(G^3)=3-\frac{6}{3n-11}.
\]
\end{lemma}

\begin{proof}

We enumerate all maximal matchings of the whole graph as follows. Choose a pendant edge $(v_3, l_{v_3})$ ($k$ choices) in a maximal matching. Then $v_3$ is covered and $(v_1,v_2)$ is also in this maximal matching. The bridge $(v_3,w_1)$ is dominated. There are $3$ ways to form this maximal matching: $(w_1,w_2)$, $(w_1,w_3)$, or $(w_2,w_3)$. Each such matching has size $3$. Hence this case contributes $k \times 3 = 3k$ maximal matchings of size $3$. Choose an edge $(v_3,v_1)$ or $(v_3,v_2)$ ($2$ choices). Then $v_3$ is covered, so all leaves and the bridge are dominated. There are also $3$ ways to form this maximal matching: $(w_1,w_2)$, $(w_1,w_3)$, or $(w_2,w_3)$. Each resulting maximal matching has size $2$. Contribution: $2 \times 3 = 6$ maximal matchings of size $2$. Choose the bridge $(v_3,w_1)$. Then $v_3$ and $w_1$ are covered. This yields a maximal matching \{$(v_1, v_2), (v_3, w_1), (w_2, w_3)$\} of size $3$. Contribution: $1$ maximal matching of size $3$.
\[
\begin{cases}
m(G^3)= 3k + 6 + 1 = 3k + 7,\\[4pt]
m'(G^3)= 3(3k) + 2(6) + 3(1) = 9k + 15.
\end{cases}
\]
Substituting $k = n-6$, we obtain the average size:
\[
avm(G^3) = \frac{9(n-6) + 15}{3(n-6) + 7} = \frac{9n - 39}{3n - 11} = 3 - \frac{6}{3n - 11}= 3 - \frac{6}{3k+7}=\dfrac{9k+15}{3k+7}.
\]
\end{proof}


\begin{lemma}\label{noleon3}
For $G \in \mathcal{C}(3, 3)$ with the core of $T_n^1(3, 3)$ (\cref{cpq}(C)), if there are no leaves attached to $v_3$ or $w_1$ in $G$, then $avm(G) > avm(G^3)$.
\end{lemma}

\begin{proof}
Let \(S=\{v_1,v_2,w_2,w_3\}\).

\noindent\textbf{Case 1.} All pendant edges are attached to one vertex in $S$.

Without loss of generality, assume that $v_1$ is attached to $k=n-6$ pendant edges.
If a pendant edge $(v_1, l_{v_1})$ is chosen (there are $k$ such choices), one may either take $(v_2,v_3)$, leaving with $3$ possibilities $(w_1, w_2)$ ($(w_1,w_3) \text{ or }(w_2,w_3)$), or take $(v_3,w_1)$ and $(w_2,w_3)$.  
Both options yield maximal matchings of size $3$; the first gives $3k$ maximal matchings, the second gives $k$ maximal matchings. If no pendant edge is in a maximal matching, choosing $(v_1,v_3)$ leaves with $3$ possibilities $(w_1, w_2)$ ($(w_1,w_3) \text{ or }(w_2,w_3)$), giving $3$ maximal matchings of size $2$. Choosing $(v_1,v_2)$ offers two sub‑options: either also take the bridge $(v_3,w_1)$ (forcing $(w_2,w_3)$, one maximal matching of size $3$), or omit the bridge and cover $w_1$ by one of its two incident edges $(w_1,w_2)$ or $(w_1,w_3)$ (giving $2$ maximal matchings of size $2$). Then we have
\[
m(G) = 4k + 6,\text{ and } m'(G) = 12k + 13.  
\]
Hence $avm(G) = (12k+13)/(4k+6) = 3 - 5/(4k+6)$.  
Comparing with $avm(G^3) = 3 - 6/(3k+7)$, the inequality $5/(4k+6) < 6/(3k+7)$ reduces to $9k+1>0$, which holds for all $k\ge1$. 

\noindent\textbf{Case 2.} All pendant edges are attached to exactly two vertices in $S$.
\begin{itemize}
\item \textit{All pendant edges are attached to  $v_1$ and $v_2$ (with the same counts as to $w_2$ and $w_3$).}
Let $a,b\ge1$ be the numbers of pendant edges on $v_1$ and $v_2$ respectively, so $a+b=k$. If $(v_1, l_{v_1})$ and $(v_2, l_{v_2})$ are chosen ($ab$ possibilities) in maximal matching, then $v_1,v_2$ are covered. Taking an edge $(w_1, w_2)$ or $(w_1,w_3)$ (two choices) yields a maximal matching of size $3$ and contributes $2ab$; taking the bridge instead forces $(w_2,w_3)$ and gives a maximal matching of size $4$, contributing $ab$. If $(v_1, l_{v_1})$ is chosen (in $a$ ways) in maximal matching but $(v_2, l_{v_2})$ is not, we must take $(v_2,v_3)$ to dominate the leaves on $v_2$. Then there are three choices $(w_1,w_2)$ (or $(w_1,w_2)$ or $(w_2, w_3)$), producing $3a$ maximal matchings of size $3$. Symmetrically, choosing only a leaf on $v_2$ gives $3b$ maximal matchings of size $3$. If no pendant edge is chosen, we must cover both $v_1$ and $v_2$ by taking $(v_1,v_2)$. Taking the bridge forces $(w_2,w_3)$ (size $3$, $1$ maximal matching); omitting the bridge forces one of $(w_1,w_2)$ or $(w_1,w_3)$ (size $2$, $2$ maximal matchings). Summing up: $m(G) = 3ab+3k+3$, $m'(G) = 10ab+9k+7$. Hence $avm(G) = 3 + \frac{ab-2}{3ab+3k+3}>avm(G^3)$.

\item \textit{All pendant edges are attached to  $v_1$ and $w_2$}. Let $a,b\ge1$ be the pendant edge counts on $v_1$ and $w_2$, with $a+b=k$. Choosing $(v_1,l_{v_1})$ and $(w_2,l_{w_2})$ in maximal matching ($ab$ ways) covers $v_1$ and $w_2$. The remaining vertices are $v_2,v_3,w_1,w_3$ form a path $P_4$: $v_2-v_3-w_1-w_3$. Maximal matchings on this path are: taking the two end edges $(v_2,v_3)$ and $(w_1,w_3)$ (size $4$, $ab$ maximal matchings) or taking the middle edge $(v_3,w_1)$ (size $3$, another $ab$ maximal matchings). Choosing only a leaf on $v_1$ ($a$ ways) to be in the maximal matching: to cover $w_2$ we may take $(w_1,w_2)$ or $(w_2,w_3)$. Taking $(w_1,w_2)$ covers $w_1$ and dominates the bridge, forcing $(v_2,v_3)$ (size $3$, $a$ maximal matchings). Taking $(w_2,w_3)$ leaves $w_1$ free and the bridge undominated; we must then cover the bridge by either $(v_2,v_3)$ or $(v_3,w_1)$, giving two possibilities (both size $3$, $2a$ matchings). Total from this case: $3a$ maximal matchings of size $3$. Symmetrically, choosing only a leaf on $w_2$ gives $3b$ maximal matchings of size $3$. If no pendant edge belongs to the maximal matching, then we must cover $v_1$ and $w_2$ using core edges. The pairs \{$(v_1,v_2), (w_1,w_2)$\}, \{$(v_1,v_3), (w_1,w_2)$\}, \{$(v_1,v_3), (w_2,w_3)$\} each give a maximal matching of size $2$ (3 maximal matchings). The combination $\{(v_1,v_2), (w_2,w_3)\}$ leaves the bridge undominated, forcing its inclusion and yielding a maximal matching of size $3$ (1 maximal matching). Summing: $m(G) = 2ab+3k+4$, $m'(G) = 7ab+9k+9$. Thus $avm(G)= 3 + \frac{ab-3}{2ab+3k+4}>avm(G^3)$.  
\end{itemize}

\noindent\textbf{Case 3.} All pendant edges are attached to three vertices in $S$.

Let pendant edges be attached to $v_1$, $v_2$ and $w_2$ with counts $a,b,c\ge1$, respectively, so that $k=a+b+c$. Choosing $(v_1,l_{v_1})$, $(v_2,l_{v_2})$ and $(w_2,l_{w_2})$ in maximal matching ($abc$ ways) covers $v_1,v_2,w_2$. The remaining vertices $v_3,w_1,w_3$ induce the edges $(v_3,w_1)$ and $(w_1,w_3)$, we must take either $(w_1,w_3)$ or $(v_3,w_1)$, both yielding a maximal matching of size $4$. Hence we obtain $2abc$ maximal matchings of size $4$. Choosing $(v_1,l_{v_1})$, $(v_2,l_{v_2})$ ($ab$ ways) forces us to cover $w_2$: taking $(w_1,w_2)$ gives size $3$, while taking $(w_2,w_3)$ forces the bridge $(v_3,w_1)$ and yields size $4$. Thus we get $ab$ maximal matchings of size $3$ and $ab$ of size $4$. Choosing $(v_1,l_{v_1})$, $(w_2,l_{w_2})$ ($ac$ ways) forces $(v_2,v_3)$ and then $(w_1,w_3)$, giving $ac$ maximal matchings of size $4$. Similarly, choosing $(v_2,l_{v_2})$, $(w_2,l_{w_2})$ ($bc$ ways) force $(v_1,v_3)$ and then $(w_1,w_3)$, giving $bc$ maximal matchings of size $4$. If exactly one pendant edge is chosen in maximal matching, suppose it is $(v_1,l_{v_1})$ ($a$ ways). Then we must cover $v_2$ and $w_2$: taking $(v_2,v_3)$ and then either $(w_1,w_2)$ or $(w_2,w_3)$ yields a maximal matching of size $3$, so $2a$ maximal matchings. Symmetrically, $(v_2,l_{v_2})$ ($b$ ways) gives $2b$ maximal matchings of size $3$ (via $(v_1,v_3)$ and either $(w_1,w_2)$ or $(w_2,w_3)$). $(w_2,l_{w_2})$ ($c$ ways) forces $(v_1,v_2)$ and then either $(v_3,w_1)$ or $(w_1,w_3)$, again size $3$, giving $2c$ maximal matchings. If no pendant edge is in maximal matching, we must cover $v_1,v_2,w_2$ using core edges. Taking $(v_1,v_2)$ together with $(w_1,w_2)$ yields a maximal matching of size $2$ (the only size‑$2$ maximal matching in this configuration). Taking $(v_1,v_2)$ together with $(w_2,w_3)$ and $(v_3,w_1)$, forcing its inclusion and giving a maximal matching of size $3$. Summing all contributions:
\[
\begin{aligned}
m(G) &= 2abc + (2ab + ab+ac+bc) + 2k + 2 = 2abc + 2ab + ac + bc + 2k + 2,\\
m'(G) &= 4(2abc + ab+ac+bc) + 3(2ab + 2k + 1) + 2\cdot1 = 8abc + 7ab + 4ac + 4bc + 6k + 5.
\end{aligned}
\]
Thus
\[
avm(G) = 4 - \frac{ab + 2k + 3}{m}>avm(G^3).
\]

\noindent\textbf{Case 4.} All four vertices of \(S\) are incident with pendant edges.

In this case no maximal matching of size \(2\) exists. Indeed, any
size-\(2\) maximal matching would have to cover all four vertices
\(v_1,v_2,w_2,w_3\), and hence it must be \(\{(v_1, v_2),(w_2,w_3)\}\); however,
this matching does not dominate the bridge \((v_3,w_1)\). Thus \(m_2(G)=0\),
and consequently
\[
avm(G)\ge 3>avm(G^3).
\]

\end{proof}

\begin{lemma}\label{lev3w1}
For $G \in \mathcal{C}(3, 3)$ with the core of $T_n^1(3, 3)$ (\cref{cpq}(C)), if $v_3$ and $w_1$ are both incident with some pendant edges in $G$, then $avm(G) > avm(G^3)$.
\end{lemma}
\begin{proof}
Let $S = \{v_1, v_2, w_2, w_3\}$. If every vertex in $S$ is incident to pendant edges, or exactly three vertices of $S$ are incident to pendant edges, then there is no maximal matching of size $2$. Moreover, if two vertices of $S$ lie in the same cycle and both are incident to pendant edges, then again no maximal matching of size $2$ exists. Consequently, we only need to consider the following three cases: (1)no vertex of $S$ is incident to pendant edges; (2) exactly one vertex of $S$ is incident to pendant edges; (3)two vertices of $S$ (in different cycles) are incident to pendant edges.
 
\noindent\textbf{Case 1.} No vertex of $S$ is incident to pendant edges.

Consider the configuration where leaves are attached only to $v_3$ and $w_1$, with $a\ge1$ leaves on $v_3$, $b\ge1$ leaves on $w_1$, and $a+b=k$. If leaves from both sides are chosen ($ab$ ways), then $v_3$ and $w_1$ are covered, giving a maximal matching of size $4$.  
If only a leaf on $v_3$ is chosen ($a$ ways), we must cover $w_1$ by taking either $(w_1,w_2)$ or $(w_1,w_3)$ (2 choices); each yields a maximal matching of size $3$, contributing $2a$.  
Symmetrically, choosing only a leaf on $w_1$ ($b$ ways) gives $2b$ maximal matchings of size $3$.  
If no leaf is chosen, we have two subcases: without the bridge, we may choose $(v_1, v_3)$ or $(v_2,v_3)$, $(w_1, w_2)$ or $(w_1,w_3)$, yielding $4$ maximal matchings of size $2$; with the bridge $(v_3,w_1)$ itself, we only choose $(v_1,v_2)$ and $(w_2,w_3)$, giving $1$ maximal matching of size $3$. Then we have
\[
m(G)= ab + 2a + 2b + 5 = ab + 2k + 5,\qquad
m'(G) = 4ab + 3(2a+2b+1) + 2\cdot4 = 4ab + 6k + 11.
\]
Comparing with the reference $\dfrac{9k+15}{3k+7}$, the difference
\[
\frac{4ab+6k+11}{ab+2k+5} - \frac{9k+15}{3k+7}
= \frac{3abk + 13ab + 2}{(ab+2k+5)(3k+7)} > 0,
\]
since $a,b\ge1$ and $k\ge2$. Hence $avm(G)>avm(G^3)$.

\noindent\textbf{Case 2.} Exactly one vertex of $S$ is incident to pendant edges.

Consider the case where leaves are attached to $v_1$ ($a$ leaves), $v_3$ ($x$ leaves), and $w_1$ ($y$ leaves).  
When leaves from three or two of these vertices are chosen, we obtain maximal matchings of size $4$: the number is $axy$ (all three) plus $ay$ (choosing $a$ and $y$) plus $xy$ (choosing $x$ and $y$); maximal matchings of size $3$ from this part come only from choosing $a$ and $x$, giving $2ax$. When exactly one leaf is chosen, we have the following. Choosing a leaf on $v_1$ ($a$ ways): we may take $(v_2,v_3)$ together with $(w_1, w_2)$ or $(w_1, w_3)$; or we may take the bridge $(v_3,w_1)$ which forces $(w_2,w_3)$; this yields $3$ possibilities, all of size $3$, contributing $3a$. Choosing a leaf on $v_3$ ($x$ ways): we must cover $v_1$ by $(v_1,v_2)$ and then have $2$ choices $(w_1, w_2)$ or $(w_1, w_3)$, giving $2x$ maximal matchings of size $3$. Choosing a leaf on $w_1$ ($y$ ways): we must cover $v_1$ and $v_3$ by $(v_1,v_3)$ (since $(v_1,v_2)$ would leave $v_3$ uncovered) and then $(w_2,w_3)$ is forced, yielding $y$ maximal matchings of size $3$. When no leaf is chosen, taking $(v_1,v_3)$ and then either $(w_1, w_2)$ or $(w_1, w_3)$ gives $2$ maximal matchings of size $2$. Additionally, the maximal matching $\{(v_1,v_2), (v_3,w_1), (w_2,w_3)\}$ has size $3$, contributing $1$ maximal matching of size $3$. Then we have
\[
\begin{cases}
m(G) = (axy + ay + xy) + 2ax + 3a + 2x + y + 2 + 1,\\[4pt]
m'(G) = 4(axy + ay + xy) + 3(2ax + 3a + 2x + y) + 2(2) + 3(1)
\end{cases}
\]
satisfy $m'-3m = axy + ay + xy - 2$. Since $a,x,y\ge1$, this difference is at least $1$, hence strictly positive. Therefore the average size is always greater than $3$ and $avm(G)>avm(G^3)$.

\noindent\textbf{Case 3.} Two vertices of $S$ ($v_1$ and $w_2$) are incident to pendant edges.

Consider the case where leaves are attached to four vertices: $v_1$ ($a$ leaves), $w_2$ ($b$ leaves), $v_3$ ($x$ leaves), and $w_1$ ($y$ leaves), with $a,b,x,y\ge1$. If the pendant edges attached to all four of these vertices, or to exactly three of them, are chosen to be in a maximal matching, then every resulting maximal matching has size 4, the number of such maximal matchings is $abxy$ (all four) plus $abx + aby + axy + bxy$ (any three). When exactly two pendant edges are chosen to be in a maximal matching, \{$(v_1,l_{v_1})$, $(w_2, l_{w_2})$, $(v_2, v_3)$, $(w_1, w_3)$\} yields $ab$ maximal matchings of size $4$, and additionally $ab$ maximal matchings of size $3$ coming from \{$(v_1,l_{v_1})$, $(w_2, l_{w_2})$, $(v_3,w_1)$\}. The pairs \{$(v_1, l_{v_1})$, $(w_1, l_{w_1})$\}, \{$(v_3, l_{v_3})$, $(w_2, l_{w_2})$\} and \{$(v_3, l_{v_3})$, $(w_1, l_{w_1})$\} each force size $4$, contributing $ay+bx+xy$ maximal matchings of size $4$. The pairs \{$(v_1, l_{v_1})$, $(v_3, l_{v_3})$\} and \{$(w_1, l_{w_1})$, $(w_2, l_{w_2})$\} each yield maximal matchings of size $3$, contributing $ax+by$ maximal matchings of size $3$. When exactly one pendant edge is chosen to be in a maximal matching, we have the following. Choosing a leaf on $v_1$ ($a$ ways): we may take $(v_2,v_3)$ together with $(w_1,w_2)$, or take the bridge $(v_3,w_1)$ together with $(w_2,w_3)$; both options give size $3$, so $2a$ maximal matchings. Similarly, a leaf on $w_2$ ($b$ ways) gives $2b$ maximal matchings of size $3$; a leaf on $v_3$ ($x$ ways) gives $x$ maximal matchings of size $3$; a leaf on $w_1$ ($y$ ways) gives $y$ maximal matchings of size $3$. When no leaf is chosen to be in a maximal matching, we must cover $v_1,w_2,v_3,w_1$. One possibility is $(v_1,v_3)$ together with $(w_1,w_2)$, which yields the unique maximal matching of size $2$ (count $1$). Another possibility is the maximal matching $\{(v_1,v_2), (v_3,w_1), (w_2,w_3)\}$, which has size $3$ (count $1$).  

Now compute $m'(G)-3m(G)$. All maximal matchings of size $3$ contribute $0$ to $m'(G)-3m(G)$. The contributions from size‑$4$ maximal matchings are $+1$ each, and the single size‑$2$ maximal matching contributes $-1$. Therefore
\[
m'(G)-3m(G) = (abxy + abx + aby + axy + bxy) + (ab + ay + bx + xy) - 1.
\]
Since each of the nine terms $abxy, abx, aby, axy, bxy, ab, ay, bx, xy$ is at least $1$, their sum is at least $9$, so $m'(G)-3m(G)> 0$. Hence the average size satisfies $avm(G) > 3$ in this configuration and $avm(G) > avm(G^3)$.
\end{proof}

\begin{lemma}\label{lev3}
For $G \in \mathcal{C}(3, 3)$ with the core of $T_n^1(3, 3)$ (\cref{cpq}(C)). If $v_3$ is incident with some pendant edges ($w_1$ is not) and at least one vertex of $S = \{v_1, v_2, w_2, w_3\}$ is also incident with some pendant edges in $G$, then $avm(G) > avm(G^3)$.
\end{lemma}
%
\begin{proof}
Without loss of generality, assume that $x(x\ge 1)$ leaves are adjacent to $v_3$. Similar to Lemma \ref{lev3w1}, we only need to consider the following cases.

\noindent\textbf{Case 1.} Two vertices of $S$ ($v_1$ and $w_2$) are incident to pendant edges. 

Consider the case where leaves are attached to the two vertices: $v_1$ ($a$ leaves), $w_2$ ($b$ leaves), with $a,b\ge1$ and total leaves $k=a+b+x$. If $(v_1, l_{v_1})$, $(v_3, l_{v_3})$ and $(w_2, l_{w_2})$ are in a maximal matching, then $(w_1,w_3)$ is also in this maximal matching. Hence we obtain $axb$ maximal matchings of size $4$. When no leaf is chosen in maximal matching. Taking $(v_1,v_3)$ together with either $(w_1, w_2)$ or $(w_2, w_3)$ produces $2$ maximal matchings of size $2$. Additionally, the maximal matching $\{(v_1,v_2), (v_3,w_1), (w_2,w_3)\}$ has size $3$, contributing $1$ maximal matching of size $3$.

If $(v_1, l_{v_1})$ and $(v_3, l_{v_3})$ are in a maximal matching ($ax$ ways), choosing $(w_1,w_2)$ or $(w_2,w_3)$ gives $2ax$ maximal matchings of size $3$. For $(v_1, l_{v_1})$ and $(w_2, l_{w_2})$ ($ab$ ways), taking $(v_2,v_3)$ and $(w_1,w_3)$ yields $ab$ maximal matchings of size 4, while taking the bridge $(v_3,w_1)$ yields $ab$ maximal matchings of size 3. For $(v_3, l_{v_3})$ and $(w_2, l_{w_2})$ ($xb$ ways), taking $(v_1,v_2)$ and $(w_1,w_3)$ yields $xb$ maximal matchings of size $4$.

Choosing $(v_1, l_{v_1})$ ($a$ ways) allows either $(v_2,v_3)$ with $(w_1, w_2)$ (or $(w_2, w_3)$) ($2$ choices) or the bridge $(v_3,w_1)$ with $(w_2,w_3)$; all give size $3$, so $3a$ maximal matchings. Choosing $(v_3, l_{v_3})$ ($x$ ways) forces us to cover $v_1$ and $w_2$: take $(v_1,v_2)$ and either right $(w_1, w_2)$ or $(w_2, w_3)$ ($2$ ways), giving $2x$ maximal matchings of size $3$. Choosing $(w_2, l_{w_2})$ ($b$ ways) yields two options: $(v_1,v_3)$ with $(w_1,w_3)$, or $(v_1,v_2)$ with the bridge $(v_3,w_1)$; both have size $3$, so $2b$ maximal matchings.

Now compute $m'(G)-3m(G)$. The only contributions come from size‑$4$ maximal matchings ($+1$ each) and the two size‑$2$ maximal matchings ($-1$ each). Therefore
\[
m'(G)-3m(G) = (axb) + (ab + xb) + (2\cdot(-1)) = axb + ab + xb - 2.
\]
Since $a,b,x\ge1$, we have $axb + ab + xb \ge 1+1+1 = 3$, hence $m'(G)-3m(G) \ge 1$. Consequently $avm(G) > avm(G^3)$ for this configuration.

\noindent\textbf{Case 2.} Two vertices of $S$ ($w_2$ and $w_3$) are incident to pendant edges.

Consider the case where leaves are attached to the two vertices: $w_2$ ($b$ leaves), $w_3$ ($c$ leaves), with $b,c\ge1$ and total leaves $k=x+b+c$. If $(v_3, l_{v_3})$, $(w_2, l_{w_2})$ and $(w_3, l_{w_3})$ are in maximal matching, then $(v_1,v_2)$ is also in this maximal matching. Hence we obtain $xbc$ maximal matchings of size $4$. When no leaf is chosen in maximal matching. Taking $(w_2,w_3)$ together with either $(v_1,v_3)$ or $(v_2,v_3)$ produces $2$ maximal matchings of size $2$. Additionally, the maximal matching $\{(v_1,v_2), (v_3,w_1), (w_2,w_3)\}$ has size $3$, contributing $1$ maximal matching of size $3$.

If $(v_3, l_{v_3})$ and $(w_2, l_{w_2})$ are in a maximal matching ($xb$ ways), choosing $(w_1,w_3)$ and $(v_1,v_2)$ gives $xb$ maximal matchings of size $4$. Symmetrically, $(v_3, l_{v_3})$ and $(w_3, l_{w_3})$ ($xc$ ways) forces $(w_1,w_2)$ and $(v_1,v_2)$, yielding $xc$ maximal matchings of size $4$. For $(w_2, l_{w_2})$ and $(w_3, l_{w_3})$ ($bc$ ways), choosing either $(v_1,v_3)$ or $(v_2,v_3)$ gives $2bc$ maximal matchings of size $3$; alternatively, choosing $(v_3,w_1)$ together with $(v_1,v_2)$ yields $bc$ maximal matchings of size $4$.

Choosing $(v_3, l_{v_3})$ ($x$ ways) forces us to cover $w_2$ and $w_3$: take $(v_1,v_2)$ and $(w_2, w_3)$, giving $x$ maximal matchings of size $3$. Choosing $(w_2, l_{w_2})$ ($b$ ways) yields two options: $(v_1,v_3)$ with $(w_1,w_3)$, or $(v_2,v_3)$ with $(w_1,w_3)$; both have size $3$, so $2b$ maximal matchings. Symmetrically, choosing $(w_3, l_{w_3})$ gives $2c$ maximal matchings of size $3$.

Now compute $m'(G)-3m(G)$. The contributions come only from size‑$4$ maximal matchings ($+1$ each) and the two size‑$2$ maximal matchings ($-1$ each). Therefore
\[
m'(G)-3m(G) = (xbc) + (xb + xc) + (bc) + (2\cdot(-1)) = xbc + xb + xc + bc - 2.
\]
Since $x,b,c\ge1$, we have $xbc + xb + xc + bc \ge 1+1+1+1 = 4$, so $m'(G)-3m(G) \ge 2 > 0$.

\noindent\textbf{Case 3.} Exactly one vertex of $S$ ($v_1$) is incident to pendant edges.

Let $v_1$ be adjacent to $a$ leaves, with $a\ge1$ and $k=x+a$. In this case, all maximal matchings have size $2$ or $3$. If $(v_1, l_{v_1})$ and $(v_3, l_{v_3})$ are in a maximal matching ($ax$ ways), then $v_1$ and $v_3$ are covered. Choosing $(w_1, w_2)$ or $(w_1, w_3)$ or $(w_2, w_3)$ gives $3ax$ maximal matchings of size $3$. When no leaf is chosen in a maximal matching. We have two types of maximal matchings. Taking $(v_1,v_3)$ together with $(w_1, w_2)$ or $(w_1, w_3)$ or $(w_2, w_3)$ yields $3$ maximal matchings of size $2$. Additionally, the maximal matching $\{(v_1,v_2), (v_3,w_1), (w_2,w_3)\}$ has size $3$, contributing $1$ maximal matching of size $3$.

If exactly one leaf is chosen in a maximal matching. For $(v_3, l_{v_3})$ ($x$ ways), taking $(v_1, v_2)$ together with $(w_1, w_2)$ or $(w_1, w_3)$ or $(w_2, w_3)$ gives  $3x$ maximal matchings of size $3$. For $(v_1, l_{v_1})$ ($a$ ways). Two possibilities exist: Taking $(v_2,v_3)$ together with $(w_1, w_2)$ or $(w_1, w_3)$ or $(w_2, w_3)$ gives $3a$ maximal matchings of size $3$. Taking $(v_3,w_1)$ together with $(w_2, w_3)$ gives $a$ maximal matchings of size $3$. 
\[
m(G) = 3ax + 3x + 4a + 4,\qquad 
m'(G) = 3(3ax + 3x + 4a + 1) + 2\cdot3 = 9ax + 9x + 12a + 9.
\]
Thus the average size is
\[
avm(G)= \frac{9ax + 9x + 12a + 9}{3ax + 3x + 4a + 4}
= 3 - \frac{3}{3ax + 3x + 4a + 4}>avm(G^3).
\]

\noindent\textbf{Case 4.}  Exactly one vertex of $S$ ($w_2$) is incident to pendant edges.

Let $w_2$ is adjacent with $b$ leaves, with $b\ge1$ and $k=x+b$. If $(v_3, l_{v_3})$ and $(w_2, l_{w_2})$ are in a maximal matching ($xb$ ways), then $(v_1,v_2)$ and $(w_1,w_3)$ are also in this maximal matching, giving $xb$ maximal matchings of size $4$. When no leaf is chosen in a maximal matching. Choosing $(v_1,v_3)$ or $(v_2,v_3)$ and $(w_1,w_2)$ or $(w_2,w_3)$ yields $4$ maximal matchings of size $2$. Additionally, the maximal matching $\{(v_1,v_2), (v_3,w_1), (w_2,w_3)\}$ contributes $1$ maximal matching of size $3$.

If exactly one leaf is chosen in a maximal matching. For $(v_3, l_{v_3})$ ($x$ ways), taking $(v_1,v_2)$ together with $(w_1,w_2)$ or $(w_2,w_3)$ yields $2x$ maximal matchings of size $3$. For $(w_2, l_{w_2})$ ($b$ ways), taking $(w_1,w_3)$ together with $(v_1,v_3)$ or $(v_2,v_3)$ yields $2b$ maximal matchings of size 3, taking $(v_1,v_2)$ together with the bridge $(v_3,w_1)$ yields $b$ maximal matchings of size 3.
\[
m(G) = xb + 2x + 3b + 5,\qquad
m'(G) = 4xb + 3(2x+3b+1) + 2\cdot4 = 4xb + 6x + 9b + 11.
\]
Thus the average size is
\[
avm(G) = \frac{4xb + 6x + 9b + 11}{xb + 2x + 3b + 5}
= 3 + \frac{xb - 4}{xb + 2x + 3b + 5}>avm(G^3).
\]
\end{proof}

\begin{lemma}\label{r=2}
For $G\in \mathcal{C}(3, 3)$ with $7$ vertices at least, if $G$ has the following core (\cref{d323}), then $avm(G)>avm(G^3)$.
\begin{figure}[h]
\begin{tikzpicture}[scale=0.8]
\node[vertex] (v3) at (-0.5,0) [label=above:$v_3$] {};
\node[vertex] (v1) at (-1.5, 0.6) [label=above:$v_1$] {};
\node[vertex] (v2) at (-1.5, -0.6)[label=below:$v_2$] {};
\node[vertex] (k) at (0.5,0) [label=above:$u$] {};
\node[vertex] (w1) at (1.5,0) [label=above:$w_1$] {};
\node[vertex] (w2) at (2.5, 0.6)[label=above:$w_2$] {};
\node[vertex] (w3) at (2.5, -0.6)[label=below:$w_3$] {};
\draw[edge] (v3) -- (v1) -- (v2) -- (v3);
\draw[edge] (v3) --  (w1);
\draw[edge] (w1) -- (w2) -- (w3) -- (w1);
\end{tikzpicture}
\caption{}\label{d323}
\end{figure}
\end{lemma}
\begin{proof}
Let $k'=n-7$ be the number of pendant edges of the graph \(G\) considered in this
lemma. Notice that the same-order extremal graph
\(G^3=T_n^1(3,3)\) has \(n-6=k'+1\) pendant edges. Therefore
\[
avm(G^3)
=
3-\frac{6}{3(k'+1)+7}
=
3-\frac{6}{3k'+10}.
\]

 In \cref{d323}, there are four maximal matchings of size $2$:
\[
 M_1=\{v_1v_3, w_1w_2\},\; M_2=\{v_1v_3, w_1w_3\},\; M_3=\{v_2v_3, w_1w_2\},\; M_4=\{v_2v_3, w_1w_3\}.
\]
All four maximal matchings universally cover $v_3$ and $w_1$. To preserve any of these size-2 maximal matchings without selecting pendant edges, the maximal matching must cover all vertices that have leaves attached; otherwise, an uncovered pendant edge could be legally added, violating maximality. Without loss of generality, we may fix the size‑$2$ core maximal matching to be $\{v_1v_3, w_1w_2\}$. Let the representative set of viable leaf-attachment vertices be $S = \{v_1, v_3, w_1, w_2\}$. If $G$ contains a support vertex that is not on the core, then every maximal matching has size at least 2. Let $d$ be the number of vertices in $S$ that are incident with at least one pendant edge. If \(k'=0\), then \(G\) is exactly the core shown in FIGURE~\ref{d323}. A direct
enumeration gives four maximal matchings of size \(2\) and six maximal
matchings of size \(3\). Hence
\[
m(G)=10,\qquad m'(G)=4\cdot 2+6\cdot 3=26,
\]
and so
\[
avm(G)=\frac{26}{10}=\frac{13}{5}.
\]
On the other hand, the same-order graph \(G^3=T_n^1(3,3)\) has one
pendant edge, and therefore
\[
avm(G^3)=3-\frac{6}{10}=\frac{12}{5}.
\]
Thus
\[
avm(G)=\frac{13}{5}>\frac{12}{5}=avm(G^3).
\]
Hence, in what follows, we may assume that \(k'\ge1\).

\noindent\textbf{Case 1.} $d=1$.
\begin{itemize}

\item \textit{If all pendant edges are attached to $v_1$ (or similarly $w_2$)}. Only $M_1$ and $M_2$ cover $v_1$, restricting $m_2 = 2$. Choosing $(v_1, l_{v_1})$ ($k'$ choices) in a maximal matching yields $k'$ maximal matchings of size 4 (specifically \{$(v_1, l_{v_1})$, $(v_2, v_3)$, $(u, w_1)$, $(w_2, w_3)$\}) and 5$k'$ maximal matchings of size 3 (specifically \{$(v_1, l_{v_1})$, $(v_2, v_3)$, $(w_1, w_2)$\text{ or }$(w_1, w_3)$\}, \{$(v_1, l_{v_1})$, $(u, v_3)$, $(w_1, w_2)$\text{ or }$(w_1, w_3)$\text{ or }$(w_2, w_3)$\}). The pure core yields $m_2 = 2$ and 5 maximal matchings of size 3. Thus, the exact counts are $m_4 = k'$, $m_3 = 5k' + 5$, and $m_2 = 2$. 
\[
m(G) = 6k' + 7, m'(G) = 4(k') + 3(5k'+5) + 2(2) = 19k' + 19.
\]
The average is $avm(G) = \frac{19k'+19}{6k'+7} = 3 + \frac{k'-2}{6k'+7}>avm(G^3)$.

\item \textit{If all pendant edges are attached to $v_3$ (or similarly $w_1$)}. All 4 size-2 maximal matchings cover $v_3$, maintaining $m_2 = 4$. Selecting $(v_3, l_{v_3})$ ($k'$ choices) in a maximal matching yields $k'$ maximal matching of size 4 (specifically \{$(v_3, l_{v_3})$, $(v_1, v_2)$, $(u, w_1)$, $(w_2,w_3)$\}) and 2$k'$ maximal matchings of size 3 (specifically \{$(v_3, l_{v_3})$, $(v_1, v_2)$, $(w_1, w_2)$ \text{ or } $(w_1,w_3)$\}). The pure core adds $m_2=4$ and 5 maximal matchings of size 3. Thus, $m_4 = k'$, $m_3 = 2k' + 5$, and $m_2 = 4$. 
\[
m(G) = 3k' + 9, m'(G) = 4(k') + 3(2k'+5) + 2(4) = 10k' + 23. 
\]
The average is $avm(G) = \frac{10k'+23}{3k'+9}>avm(G^3)$ for any integer $k \ge 1$.
\end{itemize}
\noindent\textbf{Case 2.} \(d\ge 2\).

Let \(H\) be the core shown in FIGURE~\ref{d323}.

Let
\[
X\subseteq S=\{v_1,v_3,w_1,w_2\}
\]
be the set of vertices in \(S\) that are incident with pendant edges.
Thus \(|X|=d\ge 2\). A core matching containing no pendant edge is a
maximal matching of \(G\) only if it covers every vertex in \(X\);
otherwise, an uncovered pendant edge could be added to it. In particular,
a size-\(2\) core maximal matching survives as a size-\(2\) maximal
matching of \(G\) only if it covers all vertices in \(X\).

Moreover, when \(d\ge 2\), no maximal matching of size \(2\) can contain
a pendant edge. Indeed, if a pendant edge is chosen, then the remaining
single edge cannot dominate all edges of the core \(H\) while also
dominating the pendant edges incident with the other vertices of \(X\).
Thus every size-\(2\) maximal matching of \(G\) must be a core matching.

For \(X\subseteq S\), define
\[
\alpha(X)=
\left|
\left\{
M\in \mathcal{M}(H): |M|=2,\ X\subseteq V(M)
\right\}
\right|.
\]
Also define
\[
\gamma(X)=
\left|
\left\{
M\in \mathcal{M}(H): X\subseteq V(M)
\right\}
\right|,
\]
namely, \(\gamma(X)\) is the number of core maximal matchings that cover
all vertices in \(X\). By direct enumeration of the core \(H\), the
following table is obtained:
\[
\begin{array}{c|c|c}
X & \alpha(X) & \gamma(X)\\ \hline
\{v_1,v_3\} & 2 & 6\\
\{v_1,w_1\} & 2 & 6\\
\{v_1,w_2\} & 1 & 5\\
\{v_3,w_1\} & 4 & 8\\
\{v_3,w_2\} & 2 & 6\\
\{w_1,w_2\} & 2 & 6\\ \hline
\{v_1,v_3,w_1\} & 2 & 5\\
\{v_1,v_3,w_2\} & 1 & 4\\
\{v_1,w_1,w_2\} & 1 & 4\\
\{v_3,w_1,w_2\} & 2 & 5\\ \hline
\{v_1,v_3,w_1,w_2\} & 1 & 3
\end{array}
\]
Hence, for the graph \(G\), we have
\[
m_2(G)\le \alpha(X).
\]
Moreover, the number of maximal matchings of \(G\) containing no pendant
edge is precisely \(\gamma(X)\).

We now obtain a lower bound for \(m(G)\). Consider the following three
disjoint families of maximal matchings.

First, suppose that exactly one pendant edge is chosen. For each pendant
edge, the remaining core admits at least two extensions to a maximal
matching, as can be checked directly from the core \(H\). Thus this
family contributes at least \(2k'\) maximal matchings.

Second, suppose that exactly two pendant edges are chosen. Since
\(d\ge 2\), the pendant edges are distributed among at least two distinct
vertices. The number of ways to choose two pendant edges from two
different support vertices is at least \(k'-1\). Each such choice can be
extended to at least one maximal matching. Thus this family contributes
at least \(k'-1\) maximal matchings.

Third, suppose that no pendant edge is chosen. Then the matching must be
a core maximal matching covering all vertices in \(X\). Hence this family
contributes exactly \(\gamma(X)\) maximal matchings.

Therefore,
\[
m(G)\ge 2k'+(k'-1)+\gamma(X)=3k'+\gamma(X)-1.
\]
Since every maximal matching not counted by \(m_2(G)\) has size at least
\(3\), we have
\[
m'(G)\ge 2m_2(G)+3\bigl(m(G)-m_2(G)\bigr)
=3m(G)-m_2(G).
\]
Consequently,
\[
avm(G)
=
\frac{m'(G)}{m(G)}
\ge
3-\frac{m_2(G)}{m(G)}
\ge
3-\frac{\alpha(X)}{3k'+\gamma(X)-1}.
\]

It remains to compare the above lower bound with
\[
avm(G^3)
=
3-\frac{6}{3k'+10}.
\]

If \(X=\{v_3,w_1\}\), then \(\alpha(X)=4\) and \(\gamma(X)=8\). Hence $avm(G)\ge3-\frac{4}{3k'+7}.$
Since $\frac{4}{3k'+7}<\frac{6}{3k'+10}$ is equivalent to $12k'+40<18k'+42$, which holds for every \(k'\ge 1\), we obtain $avm(G)>avm(G^3)$.

If \(|X|=2\) and \(X\ne\{v_3,w_1\}\), then $\alpha(X)\le 2$, $\gamma(X)\ge 5$. Thus $avm(G)\ge3-\frac{2}{3k'+4}$. Since $\frac{2}{3k'+4}<\frac{6}{3k'+10}$ is equivalent to $6k'+20<18k'+24$, we again have $avm(G)>avm(G^3)$.

If \(|X|=3\), then $\alpha(X)\le 2$, $\gamma(X)\ge 4$. Therefore $avm(G)\ge3-\frac{2}{3k'+3}$. Since $\frac{2}{3k'+3}<\frac{6}{3k'+10}$ is equivalent to $6k'+20<18k'+18$, which holds for every \(k'\ge 1\), we obtain $avm(G)>avm(G^3)$.

Finally, if \(X=\{v_1,v_3,w_1,w_2\}\), then $\alpha(X)=1$, $\gamma(X)=3$. Hence $avm(G)\ge3-\frac{1}{3k'+2}$. Since $\frac{1}{3k'+2}<\frac{6}{3k'+10}$ is equivalent to $3k'+10<18k'+12$, we conclude that $avm(G)>avm(G^3)$.

Combining all cases with \(d\ge 2\), we have
\[
avm(G)>avm(G^3).
\]

\end{proof}

\begin{theorem}\label{thm:C}
For any graph $G\in \mathcal{C}(p, q)$ and $G\ncong G^3$, $avm(G)>avm(G^3)$.
\end{theorem}

\begin{proof}
No graph $G\in \mathcal{C}(p, q)$ admits a maximal matching of size $2$ when the length of the path connecting the two cycles in $G$ is at least 3. No graph $G\in \mathcal{C}(p, q)$ admits a maximal matching of size $2$ when $p \ge 4$ or $q\ge 4$. If $G$ contains a support vertex that is not on the core of $\mathcal{C}_{p,q}$, then every maximal matching has size at least 3. When \(n=6\), the graph \(T_6^1(3,3)\) has no pendant edges. A direct
enumeration gives
\[
m(T_6^1(3,3))=9,\qquad m'(T_6^1(3,3))=19,
\]
and hence
\[
avm(T_6^1(3,3))=\frac{19}{9}
>
\frac{4n-11}{2n-5}\bigg|_{n=6}
=
\frac{13}{7}.
\]
Thus the case \(n=6\) does not affect the global minimum. So we only need to prove $avm(G)>avm(G^3)$ for $G\in \mathcal{C}(3, 3)$ where the length of the path connecting the two cycles is 1 or 2. Combining Lemmas \ref{avmg3}, \ref{noleon3}, \ref{lev3w1}, \ref{lev3}, and \ref{r=2}, we can get this proof. 
\end{proof}

\section{Proof of the main theorem}

\begin{proof}[Proof of Theorem~\ref{thm:main}]
By the classification of connected bicyclic graphs, every graph
\(G\in \mathcal{G}(n,n+1)\) belongs to one of the three classes
\(\mathcal{A}(p,q,l)\), \(\mathcal{B}(p,q)\), and \(\mathcal{C}(p,q)\).

If \(G\in \mathcal{A}(p,q,l)\), then Theorem~\ref{thm:A} gives
\[
avm(G)\ge avm(G^1)
=\frac{4n-11}{2n-5},
\]
where \(G^1=\theta_n^1(3,3)\). Moreover, equality holds in this case
if and only if \(G\cong G^1\).

If \(G\in \mathcal{B}(p,q)\), then Theorem~\ref{thm:B} gives
\[
avm(G)\ge avm(G^2)
=\frac{7n-27}{3n-11}.
\]
Furthermore,
\[
avm(G^2)-avm(G^1)
=
\frac{7n-27}{3n-11}-\frac{4n-11}{2n-5}
=
\frac{2(n^2-6n+7)}{(2n-5)(3n-11)}
>0
\]
for all \(n\ge 5\). Hence no graph in \(\mathcal{B}(p,q)\) can attain the
global minimum.

If \(G\in \mathcal{C}(p,q)\), then \(n\ge 6\). When \(n=6\), the only graph in
\(\mathcal{C}(p,q)\) is \(T_6^1(3,3)\), and Theorem~\ref{thm:C} gives
\[
avm(T_6^1(3,3))=\frac{19}{9}
>
\frac{4n-11}{2n-5}\bigg|_{n=6}
=
\frac{13}{7}.
\]
Hence no graph in \(\mathcal{C}(p,q)\) with \(n=6\) can attain the global minimum.

Now assume \(n\ge 7\). Then Lemma~\ref{avmg3} and Theorem~\ref{thm:C} give
\[
avm(G)\ge avm(G^3)
=
3-\frac{6}{3n-11}.
\]
Moreover,
\[
avm(G^3)-avm(G^1)
=
\left(3-\frac{6}{3n-11}\right)
-\frac{4n-11}{2n-5}
=
\frac{2(3n^2-23n+37)}{(2n-5)(3n-11)}
>0
\]
for all \(n\ge 7\). Therefore no graph in \(\mathcal{C}(p,q)\) can attain the
global minimum. Consequently,
\[
avm(G)\ge avm(G^1)
=\frac{4n-11}{2n-5}.
\]
Moreover, the above comparisons show that equality can occur only in
the class \(\mathcal{A}(p,q,l)\), and by Theorem~\ref{thm:A} this happens if and only if
\[
G\cong G^1=\theta_n^1(3,3).
\]
This completes the proof.
\end{proof}

\section*{Acknowledgement}
The author wishes to express sincere gratitude to David Guoliang Wang for his constructive feedback and valuable guidance during the final stage of writing.

\end{document}